\documentclass[
hidelinks,onefignum,onetabnum]{siamart220329}

\usepackage{lipsum}
\usepackage{amsfonts}
\usepackage{graphicx}
\usepackage{epstopdf}
\usepackage{algorithmic}
\usepackage{siunitx}
\usepackage{mathtools}

\ifpdf
\DeclareGraphicsExtensions{.eps,.pdf,.png,.jpg}
\else
\DeclareGraphicsExtensions{.eps}
\fi

\usepackage{stmaryrd}
\usepackage{bm}

\newcommand{\R}{\ensuremath{\mathbb{R}}} 
\renewcommand{\O}{\ensuremath{\Omega}} 
\newcommand{\Opm}{\ensuremath{\Omega^-\cup\Omega^+}} 

\newcommand{\ue}{{\bm{u}_\e}} 

\newcommand{\e}{\varepsilon} 
\newcommand{\ba}{\bm{a}}
\newcommand{\by}{\bm{y}}

\newcommand{\bv}{\bm{v}}
\newcommand{\bV}{\bm{V}}
\newcommand{\bu}{\bm{u}}
\newcommand{\bw}{\bm{w}}
\newcommand{\bU}{\bm{U}}
\newcommand{\bn}{\bm{n}}
\newcommand{\bb}{\bm{b}}
\newcommand{\bt}{\bm{t}}
\newcommand{\bI}{\bm{I}}
\newcommand{\bA}{\bm{A}}
\newcommand{\bB}{\bm{B}}
\newcommand{\bS}{\bm{S}}
\newcommand{\bL}{\bm{L}}
\newcommand{\bF}{\bm{F}}
\newcommand{\bG}{\bm{G}}
\newcommand{\bP}{\bm{P}}
\newcommand{\bW}{\bm{W}}
\newcommand{\bR}{\bm{R}}
\newcommand{\bK}{\bm{K}}
\newcommand{\bX}{\bm{X}}
\newcommand{\bM}{\bm{M}}
\newcommand{\be}{\bm{e}}
\newcommand{\bchi}{\bm{\chi}}
\newcommand{\bPhi}{\bm{\Phi}}
\newcommand{\bomega}{\bm{\omega}}
\newcommand{\Ae}{{\underline{\bm{A}}^\e}} 
\newcommand{\A}{{\underline{\bm{A}}}}

\newcommand{\cF}{\mathcal{F}}
\newcommand{\cP}{\mathcal{P}}
\newcommand{\cQ}{\mathcal{Q}}
\newcommand{\cT}{\mathcal{T}}
\newcommand{\cU}{\mathcal{U}}
\newcommand{\cV}{\mathcal{V}}
\newcommand{\cW}{\mathcal{W}}
\newcommand{\cY}{\mathcal{Y}}

\newcommand{\bdt}[1][]{{[\Delta t^{#1}]}}
\newcommand{\hk}{\ensuremath{\hat{\bK}}}

\newcommand{\barx}{{\bar{x}}} 
\newcommand{\baru}{{\bar{\bm{u}}}} 
 
\newcommand{\barU}{{\bar{\bm{U}}}} 

\newcommand{\ahom}{\ensuremath{a^\text{hom}}} 
\newcommand{\aahom}{\ensuremath{\underline{\bm{A}}^\text{hom}}}
\newcommand{\bhom}{\ensuremath{b^\text{hom}}} 
\newcommand{\bbhom}{\ensuremath{\underline{\bm{B}}^\text{hom}}}
\newcommand{\chom}{\ensuremath{c^\text{hom}}} 
\newcommand{\cchom}{\ensuremath{\underline{\bm{C}}^\text{hom}}}
\newcommand{\hrho}{\ensuremath{\hat{\rho}_s}}

\usepackage{ifthen}
\NewDocumentCommand{\ddx} {
	O{}
	O{}
	O{}
} {
	\ifthenelse{ \equal{#3}{} } {
		\ifthenelse{ \equal{#2}{} }
		{\ifthenelse{ \equal{#1}{} }
			{\ensuremath{\frac{\textup{d}}{\textup{d}x}}}
			{\ensuremath{\frac{\textup{d}}{\textup{d}#1}}}
		}
		{\ifthenelse{ \equal{#1}{} }
			{\ensuremath{\frac{\textup{d}^{#2}}{\textup{d}x^{#2}}}}
			{\ensuremath{\frac{\textup{d}^{#2}}{\textup{d}#1^{#2}}}}
		}
	}
	{\ifthenelse{ \equal{#2}{} }
		{\ifthenelse{ \equal{#1}{} }
			{\ensuremath{\frac{\textup{d}#3}{\textup{d}x}}}
			{\ensuremath{\frac{\textup{d}#3}{\textup{d}#1}}}
		}
		{\ifthenelse{ \equal{#1}{} }
			{\ensuremath{\frac{\textup{d}^{#2}#3}{\textup{d}x^{#2}}}}
			{\ensuremath{\frac{\textup{d}^{#2}#3}{\textup{d}#1^{#2}}}}
		}
	}	
}

\NewDocumentCommand{\pdx} {
	O{}
	O{}
	O{}
} {
	\ifthenelse{ \equal{#3}{} } {
		\ifthenelse{ \equal{#2}{} }
		{\ifthenelse{ \equal{#1}{} }
			{\ensuremath{\frac{\partial}{\partial x}}}
			{\ensuremath{\frac{\partial}{\partial#1}}}
		}
		{\ifthenelse{ \equal{#1}{} }
			{\ensuremath{\frac{\partial^{#2}}{\partial x^{#2}}}}
			{\ensuremath{\frac{\partial^{#2}}{\partial#1^{#2}}}}
		}
	}
	{\ifthenelse{ \equal{#2}{} }
		{\ifthenelse{ \equal{#1}{} }
			{\ensuremath{\frac{\partial#3}{\partial x}}}
			{\ensuremath{\frac{\partial#3}{\partial#1}}}
		}
		{\ifthenelse{ \equal{#1}{} }
			{\ensuremath{\frac{\partial^{#2}#3}{\partial x^{#2}}}}
			{\ensuremath{\frac{\partial^{#2}#3}{\partial#1^{#2}}}}
		}
	}	
}

\NewDocumentCommand{\dx} 
{
	O{}
} 
{
	\ifthenelse{ \equal{#1}{} }
	{\ensuremath{\;\textup{d}}x}
	{\ensuremath{\;\textup{d}}#1}
}

\NewDocumentCommand{\dbx} 
{
	O{}
} 
{
	\ifthenelse{ \equal{#1}{} }
	{\ensuremath{\;\textup{d}}\bar{x}}
	{\ensuremath{\;\textup{d}}#1}
}

\def\eg{e.g., }
\def\ie{i.e., }
\def\wrt{w.r.t.\ }

\newsiamremark{remark}{Remark}
\newsiamremark{hypothesis}{Hypothesis}
\crefname{hypothesis}{Hypothesis}{Hypotheses}
\newsiamthm{claim}{Claim}
\newsiamthm{assumption}{Assumption}
\newsiamthm{example}{Example}

\headers{Asymptotically proved numerical coupling of a 2D flexural porous plate with the 3D Stokes fluid }{M. Krier, J. Orlik, G. Panasenko and K. Steiner}

\title{Asymptotically proved numerical coupling of a 2D flexural porous plate with the 3D Stokes fluid\thanks{Submitted to the editors {\color{blue} DATE 28.12.2023}.
		\funding{This work was funded by the German Research Foundation under DFG-Project OR 190/6-3.}}}

\author{Maxime Krier
	\thanks{Department of Flow and Material Simulation, Fraunhofer ITWM, Kaiserslautern, Germany 
		(\email{maxime.krier@itwm.fraunhofer.de}, \email{julia.orlik@itwm.fraunhofer.de}, \email{konrad.steiner@itwm.fraunhofer.de}).}
	\and Julia Orlik
	\footnotemark[2]
	\and Grigory Panasenko
	\thanks{University of Saint-Etienne, France 
		(\email{grigory.panasenko@univ-st-etienne.fr}).}
	\and Konrad Steiner
	\footnotemark[2]}

\usepackage{amsopn}

\ifpdf
\hypersetup{
  pdftitle={Asymptotically proved numerical coupling of a 2D flexural porous plate with the 3D Stokes fluid},
  pdfauthor={M. Krier, J. Orlik, G. Panasenko and K. Steiner}
}
\fi

\begin{document}

\maketitle

\begin{abstract}
This paper presents an efficient coupling of the 3D Stokes flow interacting with an effective perforated periodic heterogeneous anisotropic 2D plate. The effective model was obtained by the asymptotic analysis in earlier works and here an effective numerical algorithm is given. By $Q_3$ or bi-cubic spacial interpolation the time-dependent problem was reduced to an algebraic system of ordinary differential equation in time. Different examples were given, demonstrating the influence of the structural plate parameters on the solution. 	
{\color{blue} }
\end{abstract}



\section{Introduction}

This paper presents an efficient 2D-3D-coupling of the Stokes flow interacting with a "stiff" perforated periodic heterogeneous plate of the thickness  and period $\e$. Under "stiff" we mean a certain contrast in the elastic properties of the plate w.r.t. the fluid viscosity, which is $\sim \e^{-3}$.\\
%
Simultaneous homogenization and dimension reduction for perforated plates, textiles and a shell in
different loading regimes was performed in 
%
 \cite{vladimir:textileBeams, GKOS, Wackerle:VonKarman, GOW, loose, GHO}.
%
The limiting macroscopic elasticity problem describes homogeneous 2D Koiter plates, \cite[Chap.3]{Panasenko:Book}, \cite[Chap.11]{Griso:Book}, \cite{GKOS}, or shells,  \cite{GHO}, or a von-Karman plate, \cite{Wackerle:VonKarman}. The plate's stiffness is given in terms of three homogenized fourth-order stiffness tensors, whose entries are determined by auxiliary elasticity problems formulated on the smallest periodic unit of the structure. In literature, these equations are usually referred to as \textit{(elasticity) cell problems}. In \cite{Wackerle:VonKarman} it was shown that the linear and non-linear von-Karman plates and in \cite{GHO} shells  share the same auxiliary cell-problems on the periodic cubes of the structure.
\\

%
%
%
%
Our work starts with  results of recent homogenization and dimension reduction from \cite{OPS, GJR}  for a Koiter plate (corresponding to the small strains) coupled with the Stokes flow. 
 The authors of both articles impose a linearized coupling condition at the fluid-structure interface, namely the continuity of velocities as well as the continuity of normal stresses. Both fluid and structure equations are formulated on fixed, time-independent domains.

Exploiting the same tools as for periodic plates and shells, 
the limit system is an 
immersed 
2D 
plate coupled with 3D Stokes flow in two simple bulk domains. 
The arising macroscopic model parameters are the three fourth-order homogenized stiffness tensors attained from the same cell problems as in the homogenization and dimension reduction of the elastic structure in \cite{GKOS}. In this paper, the cell problems are generalized to incorporate a linearized contact condition of Robin-type between individual yarns in the structure adopted from \cite{vladimir:textileBeams}.
The derived coupling conditions for the macroscopic FSI are non-standard, the couple 2D-plate with 3D bulk fluid domains. The plate's vibration is proportional to the jump of fluid stresses across the plate. The fluid velocity components tangential to the plate are vanishing, while the fluid's velocity and the plate's velocity in normal direction coincide. \\
Unintuitively, in the macroscopic limit of both FSI systems \cite{OPS, GJR}, the homogenized structure is no longer permeable, such that in particular steady-state solutions for our simulation setup may no longer exist. For this reason, we propose and investigate an extended model with an additional interface flux term obeying Darcy's law. The resulting poroelastic model shares similarities with the Stokes-Stokes system considered in \cite{bloodFlowPorousInterface} for the steady state, as well as with the Biot-Kirchhoff-plate systems in \cite{mikelic:biotDimensionReduction1} and \cite{Muha:FSI2,Muha:FSI3}. A new macroscopic model parameter, namely the structure's second-order permeability tensor, is introduced. The entries of the permeability tensor are attained from the cell problems of Darcy's law in the fluid part of the periodic unit of the structure. The assumptions on the sieve geometry, relation between the plate thickness and size of the halls to make it permeable will be discussed in \cite{Larysa} soon. Here we would like to emphasize previous well-known works on the Neumann sieve and filtration through a porous layer such as \cite{CDGO_2008} \cite{amirat}, \cite{marciniak}, \cite{mikelic}, \cite{carr}, \cite{marusic}, \cite{gomez}.

While the coupling condition with the plate's vibration remains unchanged, the velocity coupling is generalized to the fluid velocity, corrected by the plate's normal velocity, being proportional to the jump of stresses across the plate. In the limit case of vanishing permeability, the model from \cite{OPS,GJR} is recovered.

Independent of the introduction of a porous interface condition, fluid and structure equations are solely one-way coupled in the stationary case.

The well-posedness of the general problem is ensured with Galerkin methods adopted from \cite{OPS, Muha:FSI2,Muha:FSI3}. A simplified proof can be performed utilizing semigroup theory (see  \cite{modelDescription}) under frequently met restricting assumptions on the symmetry of the microscopic structure, that allow the interpretation of the new FSI model as a generalized Cauchy problem on some Hilbert space. \\

The complete numerical workflow for the simulation in the macroscopic FSI setting is presented for a textile plate and results are illustrated for three real-life woven filter samples. 
We accounted on the fiber structure being in contact with each other in the plate domain. The dimension reduction approach \cite[Sec. 3]{griso:elementaryDisplacements}  is recalled for the woven plate that allows the restriction of general elasticity problems on the porous plate-domain to 1D equations on the graph of the yarn centerlines. The method is utilized to extend the results attained from \cite{vladimir:textileBeams} to the computation of the homogenized coupling and the bending stiffness tensor. A sensitivity analysis for the influence of changing design parameters on the individual tensor entries is performed and the plausibility of attained results are discussed also in regard to the fulfillment of derived theoretical properties.

Furthermore, a numerical method for the computation of the permeability tensor is presented that utilizes a pre-implemented microscopic finite volume solver. Additional sensitivity studies for the permeability are performed and results are validated with experimental measurements. A semi-analytical expression for the case of woven filters is proposed and verified with simulation results.

For the FSI problem, a monolithic FEM solving routine is derived and implemented. The immersed plate is treated as an interior boundary with continuous velocity elements and discontinuous pressure elements. For the fluid variables, a formulation with LBB stable FE and a stabilized formulation with equal order interpolation of velocity and pressure are proposed. The stabilization method is based on the classical consistently stabilized methods for Stokes flow in simply connected domains (see \eg \cite{Bochev:TaxonomyStokesStabilizations}) extended by a stabilization term on the 2D plate interface taken from \cite{bloodFlowPorousInterface}.

For the structure equations, a formulation with conforming and a second formulation with non-conforming elements is presented and compared for a decoupled Kirchhoff plate. The conforming formulation requires $H^2$-conforming elements for which the classical Bogner-Fox-Schmit \cite{BFS} elements are chosen. A similar ansatz was recently formulated in \cite{riccardo:latticeStructures} for the interpolation and extension of the displacement of 1D lattice structures to 2D domains, if the information on the mixed derivatives is missing. The non-conforming elements require a penalized formulation based on the continuous-discontinuous Galerkin approach described in \cite{engel:CDGApproach}. Theoretical error convergence rates are verified.

%
%
%

%

The paper is organized as follows. In
\cref{sec:microModel}, the multi-scale problem formulation is given for the interaction of $\e$-periodic and -thick stiff porous plate with the Stokes flow in a 3D-channel. While, \cref{sec:macroModel} deals with its limit as $\e\to 0$ and added by the infiltration condition similar to Darcy-law. Computation of the effective plate stiffnesses is given in \cref{sec:numericalHomogenization}, while of te effective plate permeability is presented in \cref{sec:permeability}. 
The new efficient numerical algorithm for the interaction of the anisotropic 2D-plate with 3D fluid in a channel  is presented in \cref{sec:FSISolver}, 
results of numerical experiments and parameter variation are in \cref{sec:simulationResults}, and the conclusions follow in
\cref{sec:conclusions}.

\section{Multiscale problem}
\label{sec:microModel}

The non-stationary Stokes flow through a channel is considered.  In the model, the channel is separated in half along the $x_3$-direction by a thin, flexural, textile-like filter, which is fixated at its outer edges. The filter itself is of deterministic nature and posses a small period $\e$ in in-plane direction $\barx=(x_1,x_2)$, while its thickness in $x_3$-direction is of the same order as $\varepsilon$. \\
%
%
%
As an intuitive assumption for modeling of woven filters, the microscopic filter domain is assumed to be thin and periodic with an in-plane period denoted by $\e$ and a comparable thickness. It can therefore be most efficiently described by the periodic repetition of a reference cell $Y_\e^s=\e Y^s\subset \R^3$ in in-plane direction. The set $Y^s_\e$ is contained within a reference cell $\e Y$, where $Y=(0,1)^2\times(-\frac{1}{2},\frac{1}{2})$ is referred to as unit cell. The spatial variable in the reference and unit cell is denoted by $y$, respectively.

It is assumed that $Y^s_\varepsilon$ is the disjoint union of finitely many Lipschitz domains, such that the interior of the closure of $Y^s_\e$ is a connected set. Here, each Lipschitz domain can be imagined as an individual yarn. The union of shared boundaries of the Lipschitz domains is denoted by $S^c_{Y,\e}=\e S^c_Y$. It represents the contact surfaces between individual yarns. Furthermore, the complement $Y_\e^f=\e Y^f$ with $Y^f=Y_\e \setminus \overline{Y^s_\e}$ is assumed to be a connected Lipschitz domain. It is occupied with viscous fluid in the microscopic FSI model. An illustration of the introduced domains is provided in \Cref{fig:exampleReferenceCell}.
\begin{figure}[htbp]
	\centering
	\includegraphics[width=0.8\textwidth]{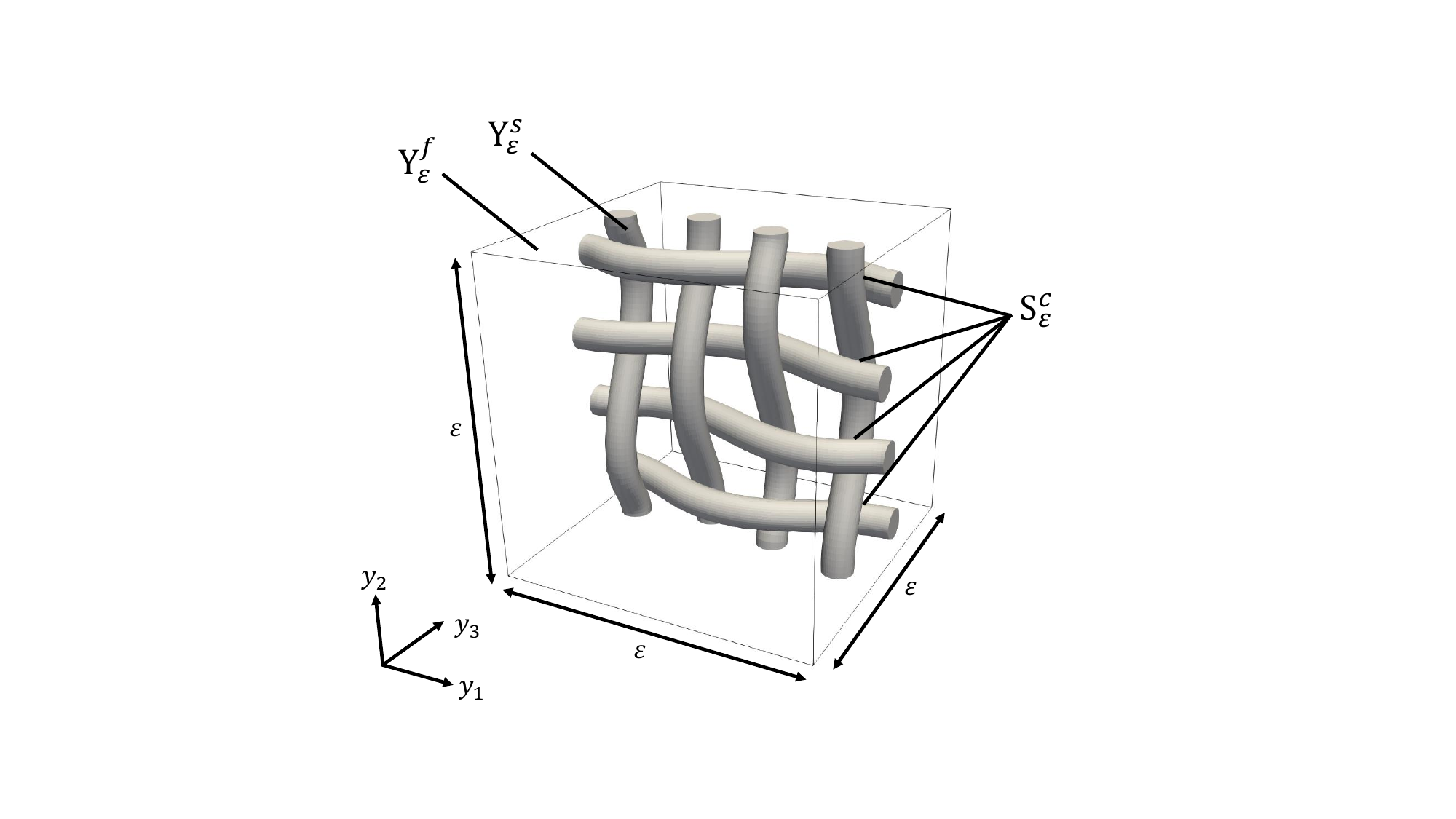}
	\caption{Example of a reference cell for a twill woven filter. {\color{blue} The notation of $S_\e^c$ has to be updated}}
	\label{fig:exampleReferenceCell}
\end{figure}

By finite periodic repetition of $Y^s_\varepsilon$, one attains a microscopic structure domain, denoted by $\O_\e^{M,s}$, which is contained within a membrane domain 
\begin{displaymath}
	\O^M_\e = (0,L_1)\times (0,L_2) \times (-\frac{\e}{2}, \frac{\e}{2}).
\end{displaymath}
The spatial variable in $\O_\e^{M}$ is denoted by $x$.

The contact surfaces, attained by periodic repetition of $S^c_{Y,\e}$, are denoted by $S_\e^c$. Furthermore, it is assumed that the filter is fixated at its outer edges, given by the set
\begin{displaymath}
	\partial^\text{fix}\O_\e^{M,s} = \partial \O^M_\e \cap \partial \O_\e^{M,s},
\end{displaymath}
assumed to be of non-zero measure and disjoint from the planes $\{x_3 = \pm\frac{\e}{2}\}$. The remaining boundary of $\O_\e^{M,s}$ is given as
\begin{displaymath}
	\partial^\text{fs}\O_\e^{M,s} = \partial \O_\e^{M,s} \setminus (\partial^\text{fix}\O_\e^{M,s} \cup S_\e^c).
\end{displaymath}

The microscopic displacement $\ue\colon (0,T)\times \O_\e^{M,s}\to\R^3$ of the filter structure is governed by linear elasticity with Robin-type contact conditions, see \cite{vladimir:textileBeams, modelDescription}. The governing system reads
\begin{equation}\label{eq:state_micro1}
	\begin{aligned}
		\rho_s\partial_{tt}\ue - \nabla\cdot\left(\Ae D(\ue) \right) 
		& = \bm{g}_\e 
		&&\quad\text{in }(0,T)\times\O_\e^{M,s}, \\
		\ue 
		& = \bm{0}
		&&\quad\text{on }(0,T)\times\partial^\text{fix}{\O_\e^{M,s}}, \\
		\left(\Ae D(\ue) \right)\bm{\eta}
		& = \bm{0}
		&&\quad\text{on }(0,T)\times\partial^\text{fs}{\O_\e^{M,s}}, \\
		\llbracket \Ae D(\ue) \rrbracket \bm{\eta}
		& = \bm{0}
		&& \quad\text{on }(0,T)\times S_\e^c, \\
		\left(\Ae D(\ue) \right)\bm{\eta} 
		& = \frac{1}{\e} \bR^\e \llbracket \ue \rrbracket 
		&& \quad\text{on }(0,T)\times S_\e^c
	\end{aligned}
\end{equation}
with solid density $\rho_s$ and initial conditions $\bu_\e(0)=\bu_0, \partial_t \bu_\e(0)=\bw_0$.

Here, $D(\bu)=\frac{1}{2}(\nabla\bu + (\nabla\bu)^T)$ denotes the symmetric gradient, $\Ae = \A(x/\e)$ with $\A\in L^\infty_\#(Y^s)^{3\times 3\times 3 \times 3}$ denotes the fourth-order material stiffness tensor and $\bR^\e= \bR(x/\e)$ with $\bR\in L^\infty_\#(S^c_\e)^{3\times 3}$ is a Robin matrix modeling contacts between individual yarns, $	\bR = \delta^{-1} \bm{\eta} \otimes \bm{\eta} 
+ \gamma_{friction} (\bI - \bm{\eta} \otimes \bm{\eta})$, where $\delta^{-1}$ and $\gamma_{friction}$ are the normal and tangential penalizing parameters. 
The term 
\begin{displaymath}
	\llbracket \bu_\e\rrbracket(x) = \lim_{\lambda\downarrow 0} \left(\bu_\e(x + \lambda\bm{\eta}) - \bu_\e(x - \lambda\bm{\eta})\right), \quad x\in S_\e^c
\end{displaymath}
for an arbitrary but fixed normal vector $\bm{\eta}$ on the interior boundary $S_\e^c$ is the jump of displacements. Hence, in the case of \textit{glued yarns}, that is $\bR^\e\to\infty$, problem \cref{eq:state_micro1} coincides with a classical elasticity problem on a single connected domain.

In the system, the product between a fourth-order tensor $\A\in\mathbb{R}^{n\times n\times n\times n}$ and a square matrix $\bB\in\mathbb{R}^{n\times n}$ is written as
\begin{displaymath}
	\A \bB = \left(\sum_{k,l=1}^n a_{ijkl} b_{kl}\right)_{i,j=1}^n \in \mathbb{R}^{n\times n}.
\end{displaymath}

The following assumptions are standard for the modeling with linear elasticity. Here and in what follows, $\bA:\bB$ denotes the standard Frobenius inner product between two square matrices.
\begin{assumption}\label{assumption:modelParameters}
	The tensor $\A=(a_{ijkl})_{i,j,k,l=1}^3$ is symmetric, \ie $a_{ijkl}=a_{jikl}=a_{klij}$ almost everywhere in $Y^s$, and coercive on the space of symmetric matrices, \ie there exists a constant $\underline{c}>0$ such that for all symmetric matrices $\bP\in\mathbb{R}^{3\times 3}$ one has $(\A\bP) : \bP\geq \underline{c} (\bP:\bP)$.
	
	Furthermore, the Robin condition matrix $\bR$ is symmetric and positive definite almost everywhere.
\end{assumption}

\Cref{assumption:modelParameters} is sufficient to ensure the existence and uniqueness of a weak solution to \cref{eq:state_micro1} under appropriate regularity of the initial conditions and of the right-hand side function $\bm{g}_\e$, see \cite{modelDescription} for details.

%
%

\section{Macroscopic model description}
\label{sec:macroModel}

In \cite{modelDescription}, a phenomenological macroscopic model for the FSI problem with non-stationary Stokes flow through a channel is proposed. In the microscopic setting, the channel is separated in half along the $x_3$-direction by a filter structure as described in the previous section, where the main direction of flow coincides with the $x_3$-direction. Linearized coupling conditions between the flow and the microscopic structure equations are prescribed, namely the continuity of velocities and of normal stresses on the fluid-structure interface $\partial^\text{fs}\O_\e^{M,s}$. Both microscopic fluid and structure equations are formulated on fixed, time-independent domains, such that the case of small filter displacements is covered. 

The model is an extension of rigorously derived macroscopic models from \cite{OPS, GJR} by the simultaneous homogenization and dimension reduction of the membrane domain $\O_\e^{M,s}$ in the scale-limit $\e\to 0$. The extension comprises of the inclusion of the linearized contact conditions between yarns from the previous section and an additional flow resistance term in the macroscopic FSI setting. \\

In the proposed macroscopic model, non-stationary Stokes flow is prescribed in two disjoint fluid domains
\begin{displaymath}
	\begin{aligned}
		\O^- &= (0,L_1) \times (0,L_2) \times (-L_3,0), \\
		\O^+ &= (0,L_1) \times (0,L_2) \times (0,L_3),
	\end{aligned}
\end{displaymath}

that is
\begin{equation}\label{eq:stokes_macro}
	\begin{aligned}
		\rho_f \partial_t \bv^\pm - 2\mu \nabla\cdot D(\bv^\pm) + \nabla p^\pm 
		&= \bm{f}^\pm
		&&\quad \text{in }(0,T)\times \O^\pm, \\
		\nabla\cdot \bv^\pm
		&= 0
		&&\quad \text{in }(0,T)\times \O^\pm
	\end{aligned}
\end{equation}
with fluid density $\rho_f$, dynamic viscosity $\mu$ and some volume force density $\bm{f}^\pm$. Moreover, the entire model domain is denoted by
\begin{displaymath}
	\O = (0,L_1) \times (0,L_2) \times (-L_3,L_3).
\end{displaymath}

The Stokes equations \cref{eq:stokes_macro} are accompanied by Dirichlet and zero-stress boundary conditions on the bottom, top and lateral boundaries
\begin{displaymath}
	\begin{aligned}
		\partial^\text{in}\O &= (0,L_1)\times (0,L_2) \times \{-L_3\}, \\
		\partial^\text{out}\O &= (0,L_1)\times (0,L_2) \times \{L_3\}, \\
		\partial^\text{no-slip}\O &= \partial\O \setminus (\partial^\text{in}\O \cup \partial^\text{out}\O).
	\end{aligned}
\end{displaymath}
The boundary conditions of choice read
\begin{equation}\label{eq:stokes_macro_boundary}
	\begin{aligned}
		\bv^-
		&= \bm{0}
		&&\quad \text{on }(0,T)\times \partial^\text{in}\O, \\
		(2\mu D(\bv^+) - p\bI)\be_3 
		&= \bm{0}
		&&\quad \text{on }(0,T)\times \partial^\text{out}\O, \\
		\bv^\pm
		&= \bm{0}
		&&\quad \text{on }(0,T)\times \partial^\text{no-slip}\O,
	\end{aligned}
\end{equation}
where $\bI$ is the $3\times 3$ unit matrix and $\be_i$ denotes the $i$-th unit vector. The inflow condition is chosen as zero for simplicity, otherwise additional regularity and extension properties of the inflow condition are required that enable the lifting of the respective solution space. \\

The two fluid domains are separated by the interior boundary
\begin{displaymath}
	\Sigma = (0,L_1) \times (0,L_2) \times \{0\},
\end{displaymath}
on which the fluid velocity is assumed to be continuous, that is $\bv^-\vert_\Sigma = \bv^+\vert_\Sigma$. The in-plane variable on $\Sigma$ is denoted by $\barx=(x_1,x_2)$.

The interface represents the mean-plane of the filter structure, whose in-plane displacement $\baru$ and outer-plane deflection $u_3$ are governed by the Kirchhoff plate equations
\begin{equation}\label{eq:plate_macro}
	\begin{aligned}
		- \nabla_\barx \cdot (\aahom D_\barx(\baru)
		+ \bbhom \nabla^2_\barx 	u_3)
		& = \bm{0} 
		&& \text{on }(0,T)\times\Sigma, \\
		\hrho \partial_{tt}u_3 + \nabla^2_\barx : (\bbhom D_\barx (\baru)
		+ 	\cchom \nabla^2_\barx u_3)
		& =
		\llbracket 2\mu D(\bv) - p\bI\rrbracket \be_3 \cdot \be_3 + g_3
		&& \text{on }(0,T)\times\Sigma
	\end{aligned}
\end{equation}
with clamped boundary conditions
\begin{equation}
	\begin{aligned}
		\baru &= \bm{0} && \text{on }(0,T)\times\partial\Sigma, \\
		u_3 = \nabla_\barx u_3\cdot\bm{\eta} &= 0 && \text{on }(0,T)\times\partial\Sigma.
	\end{aligned}
\end{equation}

Here, the expression
\begin{displaymath}
	\llbracket 2\mu D(\bv) - p\bI\rrbracket = (2\mu D(\bv^+) - p^+\bI)\vert_\Sigma - (2\mu D(\bv^-) - p^-\bI)\vert_\Sigma
\end{displaymath}
denotes the jump of fluid stresses. The operators $\nabla_\barx, D_\barx,\nabla^2_\barx$ are the respective differential operators with respect to the in-plane variables and $g_3$ is some surface force density. \\

Furthermore, the entries of the homogenized fourth-order stiffness tensors $\aahom,\bbhom,\cchom\in\mathbb{R}^{2\times2\times2\times2}$ are attained by averaging of elasticity cell solutions $\bchi_{ij}^M$ and $\bchi_{ij}^B, i,j=1,2$, reading
\begin{equation}\label{eq:homogenizedTensors}
	\begin{aligned}
		\ahom_{ijkl} 
		&=
		\frac{1}{\vert Y^s\vert}
		\left[
		\left(\A \left(D(\bchi_{ij}^M) + \bM^{ij}\right), D(\bchi_{kl}^M) + \bM^{kl} \right)_{Y^s}
		+ \left(\bR \llbracket\bchi_{ij}^M\rrbracket, \llbracket\bchi_{kl}^M\rrbracket \right)_{S^c}
		\right], \\
		\bhom_{ijkl} 
		&=  
		\frac{1}{\vert Y^s\vert}
		\left[
		\left(\A \left(D(\bchi_{ij}^B) - y_3\bM^{ij}\right), D(\bchi_{kl}^M) + \bM^{kl} \right)_{Y^s}
		+ \left(\bR \llbracket\bchi_{ij}^B\rrbracket, \llbracket\bchi_{kl}^M\rrbracket \right)_{S^c}
		\right], \\
		\chom_{ijkl} 
		&=  
		\frac{1}{\vert Y^s\vert}
		\left[
		\left(\A \left(D(\bchi_{ij}^B) - y_3\bM^{ij}\right), D(\bchi_{kl}^B) - y_3\bM^{kl} \right)_{Y^s}
		+ \left(\bR \llbracket\bchi_{ij}^B\rrbracket, \llbracket\bchi_{kl}^B\rrbracket \right)_{S^c}
		\right]
	\end{aligned}
\end{equation}
for $i,j,k,l\in\{1,2\}$ and $\bM^{ij}=\frac{1}{2}(\be_i\otimes\be_j + \be_j\otimes \be_i)\in\mathbb{R}^{3\times 3}$.

The cell solutions solve so called membrane and bending cell problems. In variational formulation, these are to find $\bchi_{ij}^{M,B}\in H^1_{\#,0}(Y^s)^3$ such that
\begin{equation}\label{eq:cellProblemStructure}
	\begin{aligned}
		& \left(\A \left(D(\bchi_{ij}^M) + \bM^{ij}\right), D(\bX) \right)_{Y^s}
		+ \left(\bR \llbracket\bchi_{ij}^M\rrbracket, \llbracket\bX\rrbracket \right)_{S^c}
		= 0, \\
		& \left(\A \left(D(\bchi_{ij}^B) - y_3\bM^{ij}\right), D(\bX) \right)_{Y^s}
		+ \left(\bR \llbracket\bchi_{ij}^B\rrbracket, \llbracket\bX\rrbracket \right)_{S^c}
		= 0
	\end{aligned}
\end{equation}
for all $\bX\in H^1_{\#,0}(Y^s)^3$. Here, $H^1_{\#,0}(Y^s)$ is the broken Sobolev space of $Y$-periodic functions, that is functions whose restrictions to the individual Lipschitz domains, that $Y^s$ is comprised of, are element of the usual Sobolev space, and which are additionally 1-periodic in in-plane direction, with vanishing mean-value in $Y^s$. Since $\bM^{ij}=\bM^{ji}$, one can verify that there are a total of six independent cell problems. 

The tensors $\aahom,\bbhom,\cchom$ are commonly referred to as extensional, coupling and bending stiffness tensor. Formally speaking, the entries of $\aahom$ determine the resistance of the structure to in-plane loads, such as applied tension and shearing, while the entries of $\cchom$ describe the resistance to bending and torsional loads. Additional coupling between in-plane strain and outer-plane bending is introduced by the entries of $\bbhom$. \\

Lastly, the macroscopic model parameter
\begin{displaymath}
	\hrho = \frac{\delta}{\vert Y_\e\vert}\int_{Y^s_\e}\rho_s\dx[y]
\end{displaymath}
is the averaged solid density $\rho_s$, with $\delta$ denoting the characteristic thickness of the structure. \\

As an additional coupling condition between fluid equations \cref{eq:stokes_macro} and structure equations \cref{eq:plate_macro}, flow-resistivity is modeled by a Darcy-interface term
\begin{equation}
	\mu \delta \bK^{-1} (\bv^+ - \partial_t u_3 \be_3) 
	= \llbracket 2\mu D(\bv) - p\bI\rrbracket \be_3
	\quad \text{on }(0,T)\times\Sigma
\end{equation}
with resistivity tensor $\bK^{-1}\in\mathbb{R}^{3\times 3}$, chosen as the inverse of the permeability tensor $\bK$. The entries of $\bK$ are given by
\begin{equation}\label{eq:permeabilityTensor}
	k_{ij} = \frac{1}{\vert Y^f\vert}(\nabla \bomega_i, \nabla \bomega_j)_{Y^f},
\end{equation}
where $\bomega_i,i\in\{1,2,3\}$ are solution to the Darcy fluid cell problems. In variational formulation, these are to find 
\begin{displaymath}
	\bomega_i\in H^1_\text{per,div}(Y^f)=\{\bW\in H^1(Y^f)^3: \bW\text{ is periodic}, \nabla\cdot\bW=0, \bW=\bm{0}\text{ on }\partial Y^s\},
\end{displaymath}
such that
\begin{equation}\label{eq:cellProblemsFluid}
	(\nabla\bomega_i, \nabla \bW)_{Y^f} = (\be_i, \bW)_{Y^f}
\end{equation}
for all $\bW\in H^1_\text{per,div}(Y^f)$. For the extreme case $\bK\to\bm{0}$, the interface $\Sigma$ becomes impermeable and the normal fluid velocity component coincides with the normal velocity of the plate. The tangential fluid velocity components vanish. One attains the FSI model derived in \cite{GJR}. For the other case $\bK\to{\infty}$, the interface is no longer seen by the fluid and the jump of fluid stresses vanishes. One attains regular Stokes flow in the entire domain $\O$.

For easier notation, the variable $\hk=\mu^{-1}\delta^{-1}\bK$ is introduced. \\

Summarizing, the macroscopic FSI problem reads
\begin{equation}\label{eq:system_macro}
	\begin{aligned}
		\rho_f \partial_t \bv^\pm - 2\mu \nabla\cdot D(\bv^\pm) + \nabla p^\pm 
		&= \bm{f}^\pm
		&& \text{in }(0,T)\times \O^\pm, \\
		\nabla\cdot \bv^\pm
		&= 0
		&& \text{in }(0,T)\times \O^\pm, \\
		\bv^-
		&= \bm{0}
		&& \text{on }(0,T)\times \partial^\text{in}\O, \\
		(2\mu D(\bv^+) - p\bI)\be_3 
		&= \bm{0}
		&& \text{on }(0,T)\times \partial^\text{out}\O, \\
		\bv^\pm
		&= \bm{0}
		&& \text{on }(0,T)\times \partial^\text{no-slip}\O, \\
		\bv^-
		&= \bv^+
		&& \text{on }(0,T)\times \Sigma, \\
		- \nabla_\barx \cdot (\aahom D_\barx(\baru)
		+ \bbhom \nabla^2_\barx 	u_3)
		& = \bm{0} 
		&& \text{on }(0,T)\times\Sigma, \\
		\hrho \partial_{tt}u_3 + \nabla^2_\barx : (\bbhom D_\barx (\baru)
		+ 	\cchom \nabla^2_\barx u_3)
		& =
		\llbracket 2\mu D(\bv) - p\bI\rrbracket \be_3 \cdot \be_3 + g_3
		&& \text{on }(0,T)\times\Sigma, \\
		\hk^{-1} (\bv^+ - \partial_t u_3 \be_3) 
		&= \llbracket 2\mu D(\bv) - p\bI\rrbracket \be_3
		&& \text{on }(0,T)\times\Sigma, \\
		\baru 
		&= \bm{0} 
		&& \text{on }(0,T)\times\partial\Sigma, \\
		u_3 = \nabla_\barx u_3\cdot\bm{\eta} 
		&= 0 
		&& \text{on }(0,T)\times\partial\Sigma,
	\end{aligned}
\end{equation}
accompanied with the initial conditions $\bv^\pm(0)=\bm{0},u_3(0)=\partial_t u_3(0)=0$. The steady-state formulation of System \cref{eq:system_macro} consists of the Stokes-Stokes coupling
\begin{equation}\label{eq:system_macro_stationaryFluid}
	\begin{aligned}
		- 2\mu \nabla\cdot D(\bv^\pm) + \nabla p^\pm 
		&= \bm{f}^\pm
		&& \text{in }\O^\pm, \\
		\nabla\cdot \bv^\pm
		&= 0
		&& \text{in }\O^\pm, \\
		\bv^-
		&= \bv^\text{in}
		&& \text{on }\partial^\text{in}\O, \\
		(2\mu D(\bv^+) - p\bI)\be_3 
		&= \bm{0}
		&& \text{on }\partial^\text{out}\O, \\
		\bv^\pm
		&= \bm{0}
		&& \text{on }\partial^\text{no-slip}\O, \\
		\bv^-
		&= \bv^+
		&& \text{on }\Sigma, \\
		\hk^{-1} \bv^+ 
		&= \llbracket 2\mu D(\bv) - p\bI\rrbracket \be_3
		&& \text{on }\Sigma,
	\end{aligned}
\end{equation}
one-way coupled to the Kirchhoff plate
\begin{equation}\label{eq:system_macro_stationaryPlate}
	\begin{aligned}
		- \nabla_\barx \cdot (\aahom D_\barx(\baru)
		+ \bbhom \nabla^2_\barx 	u_3)
		& = \bm{0} 
		&& \text{on }\Sigma, \\
		\nabla^2_\barx : (\bbhom D_\barx (\baru)
		+ 	\cchom \nabla^2_\barx u_3)
		& =
		\llbracket 2\mu D(\bv) - p\bI\rrbracket \be_3 \cdot \be_3 + g_3
		&& \text{on }\Sigma, \\
		\baru 
		&= \bm{0} 
		&& \text{on }\partial\Sigma, \\
		u_3 = \nabla_\barx u_3\cdot\bm{\eta} 
		&= 0 
		&& \text{on }\partial\Sigma.
	\end{aligned}
\end{equation}
The Stokes-Stokes coupling \cref{eq:system_macro_stationaryFluid} is actually reminiscent of the system presented in \cite{bloodFlowPorousInterface} for the modeling of immersed, rigid stents in blood flow. The cited model is based on classical Stokes-Sieve problems analyzed in \cite{conca:flowSieves1}. In the actual analysis paper \cite{Larysa}, the main parameter and bounds on the relation between the obstacle's thickness and curvature to the distance between them will be found, to make the sieve permeable or non-permeable.

For completeness, some fundamental results are recalled from literature, that suffice for the well-posedness of the system, see \cite{modelDescription} and references therein. The discussion starts with the homogenized structure.
\begin{lemma}\label{proposition:cellProblemsStructure}
	For each $i,j\in\{1,2\}$, there exists a unique cell solution $\bchi_{ij}^{M,B}\in H^1_{\#,0}(Y^s)^3$ to the cell problems \cref{eq:cellProblemStructure}, respectively.
\end{lemma}

As a consequence, one can verify the following lemma, see also Theorem 2 in \cite{GKOS}.
\begin{lemma}\label{proposition:stiffnessTensorsCoercive}
	The homogenized stiffness tensors given by the expressions \cref{eq:homogenizedTensors} are well-defined. The induced bilinear form
	\begin{displaymath}
		\begin{aligned}
			\ahom((\baru,u_3),(\barU,U_3)) &= 
			(\aahom D_\barx(\baru), D_\barx(\barU))_\Sigma 
			+ (\bbhom D_\barx(\baru), \nabla^2_\barx U_3)_\Sigma \\
			&\quad + (\bbhom \nabla^2_\barx u_3,D_\barx(\barU))_\Sigma 
			+ (\cchom\nabla^2_\barx u_3, \nabla^2_\barx U_3)_\Sigma
		\end{aligned}
	\end{displaymath}
	is continuous and bounded on $H^1_0(\Sigma)^2\times H^2_0(\Sigma)$. The induced norm \begin{displaymath}
		\Vert (\baru, u_3)\Vert_\text{hom}^2 = \ahom((\baru,u_3),(\baru,u_3))
	\end{displaymath}
	is equivalent to the standard norm on $H^1_0(\Sigma)^2\times H^2_0(\Sigma)$. The tensors $\aahom$ and $\cchom$ share the same symmetry properties as $\A$ and are coercive on the space of symmetric matrices.
\end{lemma}

Additionally, one can find the proof of the following statement on the permeability tensor \eg in Chapter 7 of \cite{sanchezPalencia:Book}.
\begin{proposition}\label{proposition:cellProblemsFluid}
	For each $i\in \{1,2,3\}$, the fluid cell problems \cref{eq:cellProblemsFluid} have a unique solution $\bomega_i\in H^1_\text{per,div}(Y^f)$. The expressions \cref{eq:permeabilityTensor} are well-defined and the resulting permeability tensor $\bK\in\mathbb{R}^{3\times 3}$ is symmetric and positive definite.
\end{proposition}

The statements are sufficient to verify the existence of solutions to the presented FSI problem, \eg by a standard Galerkin approach. A detailed proof is given in \cite{modelDescription}.
\begin{proposition}
	Let $\bm{f}^\pm \in L^2((0,T), L^2(\O^\pm)^3), g_3\in L^2((0,T), L^2(\Sigma))$. Let further the assumptions of \Cref{proposition:cellProblemsStructure} be satisfied. Then the system \cref{eq:system_macro} has a unique pressure free solution $(\bv,\baru,u_3), \bv\vert_{\O^\pm}=\bv^\pm$ with
	\begin{displaymath}
		\begin{aligned}
			\bv &\in 
			L^2((0,T),\mathcal{V}_\text{div}) \cap L^\infty((0,T),L^2(\Opm)^3), \\
			\baru &\in 
			L^2((0,T),H_0^1(\Sigma)^2), \\
			u_3 &\in 
			L^\infty((0,T),H^2_0(\Sigma)) \cap W^{1,\infty}((0,T),L^2(\Sigma)),
		\end{aligned}
	\end{displaymath}
	where
	\begin{align*}
		\mathcal{V}_\text{div}
		&=\{\bv\in \cV: \nabla\cdot \bv=0\text{ in }\Opm \}, \\
		\mathcal{V}
		&=\{\bv\in H^1(\Opm)^3: \bv=\bm{0}\text{ on }\partial^\text{in}\O\cup \partial^\text{no-slip}\O\}.
	\end{align*}
\end{proposition}

One can additionally verify that the well-posedness of the FSI system is still granted when one switches from constant stiffness and permeability tensors to the natural choice of tensors with $L^\infty$-regularity on $\Sigma$. For this purpose it is necessary to assume that the coercivity and symmetry from \Cref{proposition:stiffnessTensorsCoercive} and \Cref{proposition:cellProblemsFluid} remain valid for the generalized tensors .

Lastly, by inspecting the plate equations \cref{eq:plate_macro}, it is clear that $\baru$ vanishes whenever the coupling stiffness tensor $\bbhom$ is zero. In fact, this latter condition is frequently met under symmetry assumptions on the structure $Y^s$ and the model parameters $\A$ and $\bR$, see Lemma 6.9 in \cite{Wackerle:VonKarman}. Hence, the main displacement variable of interest is the plate's deflection.

\section{Computation of the effective textile coefficients}
\label{sec:numericalHomogenization}

In this section, an overview on the numerical computation of the homogenized stiffness tensors $\aahom$, $\bbhom$, $\cchom$ is given.
For the yarn structures in mind, an efficient dimension reduction approach with 1D beam finite elements, generalized by the incorporation of contact conditions, is presented in \cite{vladimir:textileBeams} for the computation of $\aahom$. The extension to the computation of $\bbhom$ and $\cchom$ is presented here. \\

For implementation purposes, it proves to be beneficial to introduce the perturbation functions $\bS_{ij}^{M,B}\in C^\infty(Y)^3$ as
\begin{equation}\label{eq:perturbationChoices}
	\begin{gathered}
		\bS_{11}^M(y) = \begin{pmatrix}
			y_1 \\ 0 \\ 0
		\end{pmatrix}, \quad
		\bS_{12}^M(y) = \frac{1}{2}\begin{pmatrix}
			y_2 \\ y_1 \\ 0
		\end{pmatrix}, \quad
		\bS_{22}^M (y) = \begin{pmatrix}
			0 \\ y_2 \\ 0
		\end{pmatrix}, \\
		\bS_{11}^B(y) = \frac{1}{2}\begin{pmatrix}
			-2y_1 y_3 \\ 0 \\ y_1^2
		\end{pmatrix}, \quad
		\bS_{12}^B(y) = \frac{1}{2}\begin{pmatrix}
			-y_2 y_3 \\ -y_1 y_3\\ y_1 y_2
		\end{pmatrix}, \quad
		\bS_{22}^B(y) = \frac{1}{2}\begin{pmatrix}
			0 \\ -2y_2 y_3 \\ y_2^2
		\end{pmatrix},
	\end{gathered}
\end{equation}
which are chosen as analytical solutions of the differential equations
\begin{equation}\label{eq:perturbationFunctions_PDE}
	D(\bS_{ij}^M) = \bM^{ij}, 
	\quad D(\bS_{ij}^B) = -y_3 \bM^{ij} \quad \text{in }Y^s,
	\quad \llbracket\bS_{ij}^{M,B}\rrbracket=\bm{0} \quad \text{on }S^c
\end{equation}
for $i,j=1,2$.
The choice of $\bS_{ij}^{M,B}$ with the stated properties is not uniquely determined but every function satisfying \cref{eq:perturbationFunctions_PDE} is suitable for what follows. 

By defining the augmented cell solutions $\bm{m}_{ij}^{M,B} = \bchi_{ij}^{M,B} + \bS_{ij}^{M,B} \in H^1(Y^s)^3$, one attains the equivalent cell problem formulations
\begin{equation}\label{eq:corrector_problems_reformulation}
	\left(\A D(\bm{m}_{ij}^{M,B}), D(\bX) \right)_{Y^s}
	+ \left(\bR \llbracket \bm{m}_{ij}^{M,B} \rrbracket, \llbracket\bX \rrbracket \right)_{S^c}
	= 0
\end{equation}
for all $\bX\in H^1_{\#}(Y^s)^3$, with the generalized periodicity condition that $\bm{m}_{ij}^{M,B}-\bS_{ij}^{M,B}$ are $Y$-periodic. 
The solution of the above formulation is unique up to an additive constant, since the vanishing mean value in the solution space is dropped.
It is intuitive to interpret the augmented cell solutions as actual displacement fields on $Y^s$.

The computation of the homogenized tensor entries \eqref{eq:homogenizedTensors} with the augmented cell solutions in the continuous setting becomes
\begin{equation}\label{eq:homogenizedTensors_augmented}
	\begin{aligned}
		\ahom_{ijkl} &= 
		\frac{1}{\vert Y^s\vert}
		\left[ 
		\left(\A D(\bm{m}_{ij}^M), D(\bm{m}_{kl}^M) \right)_{Y^s}
		+ \left(\bR \llbracket\bm{m}_{ij}^M \rrbracket, \llbracket\bm{m}_{kl}^M \rrbracket \right)_{S^c}
		\right], \\
		\bhom_{ijkl} &= 
		\frac{1}{\vert Y^s\vert} 
		\left[
		\left(\A D(\bm{m}_{ij}^B), D(\bm{m}_{kl}^M)\right)_{Y^s}
		+ \left(\bR \llbracket\bm{m}_{ij}^B\rrbracket, \llbracket\bm{m}_{kl}^M \rrbracket  \right)_{S^c}
		\right], \\
		\chom_{ijkl} &= 
		\frac{1}{\vert Y^s\vert} 
		\left[
		\left(\A D(\bm{m}_{ij}^B), D(\bm{m}_{kl}^B) \right)_{Y^s} 
		+ \left(\bR \llbracket\bm{m}_{ij}^B \rrbracket, \llbracket\bm{m}_{kl}^B \rrbracket \right)_{S^c}
		\right].
	\end{aligned}
\end{equation}
As can be seen, the uniqueness of $\bm{m}_{ij}^{M,B}$ up to an additive constant is sufficient in \cref{eq:homogenizedTensors_augmented}, since the functions only appear in gradient and jump terms. \\

Generally speaking, under the assumption of an appropriate choice of finite elements, the discrete formulation of the augmented cell problems \cref{eq:corrector_problems_reformulation} are the linear systems
\begin{displaymath}
	\bS \bm{m}_{ij}^{M,B} = \bm{0} \text{ + generalized $Y$-periodic boundary conditions} ,
\end{displaymath}
where the stiffness matrix $\bS$ encodes the bilinear form
\begin{displaymath}
	a(\bchi,\bX) =  \left(\A D(\bchi), D(\bX) \right)_{Y^s}
	+ \left(\bR \llbracket\bchi\rrbracket, \llbracket\bX\rrbracket\right)_{S^c}
\end{displaymath}
and $\bm{m}_{ij}^{M,B}$ are the respective DOF vectors, denoted with the same symbols as their continuous counterparts. The underlying global stiffness matrix $\bS$ remains the same for each cell problem and hence has to be assembled only once. Numerical uniqueness is attained by fixing the value of an interior DOF.

With this notation, the discrete form of the expressions \cref{eq:homogenizedTensors_augmented} reads
\begin{displaymath}
	\begin{aligned}
		\ahom_{ijkl} &= \frac{1}{\vert Y^s\vert} \left(\bm{m}_{kl}^M\right)^T \bm{S} \bm{m}_{ij}^M, \\
		\bhom_{ijkl} &= \frac{1}{\vert Y^s\vert} \left(\bm{m}_{kl}^M\right)^T \bm{S} \bm{m}_{ij}^B, \\
		\chom_{ijkl} &= \frac{1}{\vert Y^s\vert} \left(\bm{m}_{kl}^B\right)^T \bm{S} \bm{m}_{ij}^B.
	\end{aligned}
\end{displaymath}

We recall from \cite[Lemma 6.9]{Wackerle:VonKarman}, the sufficient conditions  for a vanishing coupling stiffness tensor. The stated conditions are met for typical woven filters made out of an homogeneous, isotropic material.

\begin{proposition} \label{proposition:vanishingBHom}
Let $Y^s$ be symmetric \wrt the planes $\{y_1=\frac{1}{2}\}$ and $\{y_2=\frac{1}{2}\}$ in the sense that the transformations
	\begin{align*}
	&{\bf T}_1 \colon Y^s\to Y^s,\quad  
	y\mapsto (1-y_1) \be_1 + y_2 \be_2 + y_3 \be_3, \\
	&{\bf T}_2 \colon Y^s\to Y^s,\quad  
	y\mapsto y_1 \be_1 + (1-y_2) \be_2 + y_3 \be_3
	\end{align*}
	are well-defined. 
	
	Let $\tilde{Y}^s$ denote the restriction of $Y^s$ to a quarter of the unit cell $(0,\frac{1}{2})^2\times (-\frac{1}{2},\frac{1}{2})$. Assume that $\tilde{Y}^s$ is rotational-symmetric \wrt the diagonal axis $\{y_1=y_2,y_3=0\}$ and $\{y_1=y_2=\frac{1}{4}\}$ in the sense that the transformations
	\begin{align*}
	&{\bf T}_3 \colon \tilde{Y}^s \to \tilde{Y}^s,\quad  
	y\mapsto y_2 \be_1 + y_1 \be_2 - y_3 \be_3, \\
	& {\bf T}_4 \colon \tilde{Y}^s \to \tilde{Y}^s,\quad  
	y\mapsto \left(\frac{1}{2} - y_2\right) \be_1 + y_1 \be_2 + y_3 \be_3
	\end{align*}
	are well-defined.
	
	Then we have $\bbhom=\bm{0}$, as well as the additional symmetry
	\begin{equation*}
	\ahom_{1111}=\ahom_{2222}, \quad \chom_{1111}=\chom_{2222}.
	\end{equation*}
\end{proposition}

\begin{example}
	As the standard example, we consider a plain woven filter. The domains of the proposition are illustrated in Figure \ref{fig:symmetryExample}.
	\begin{figure}[H] 
		\centering
		\includegraphics[trim={0cm 4cm 0cm 0cm}, clip, width=0.8\textwidth]{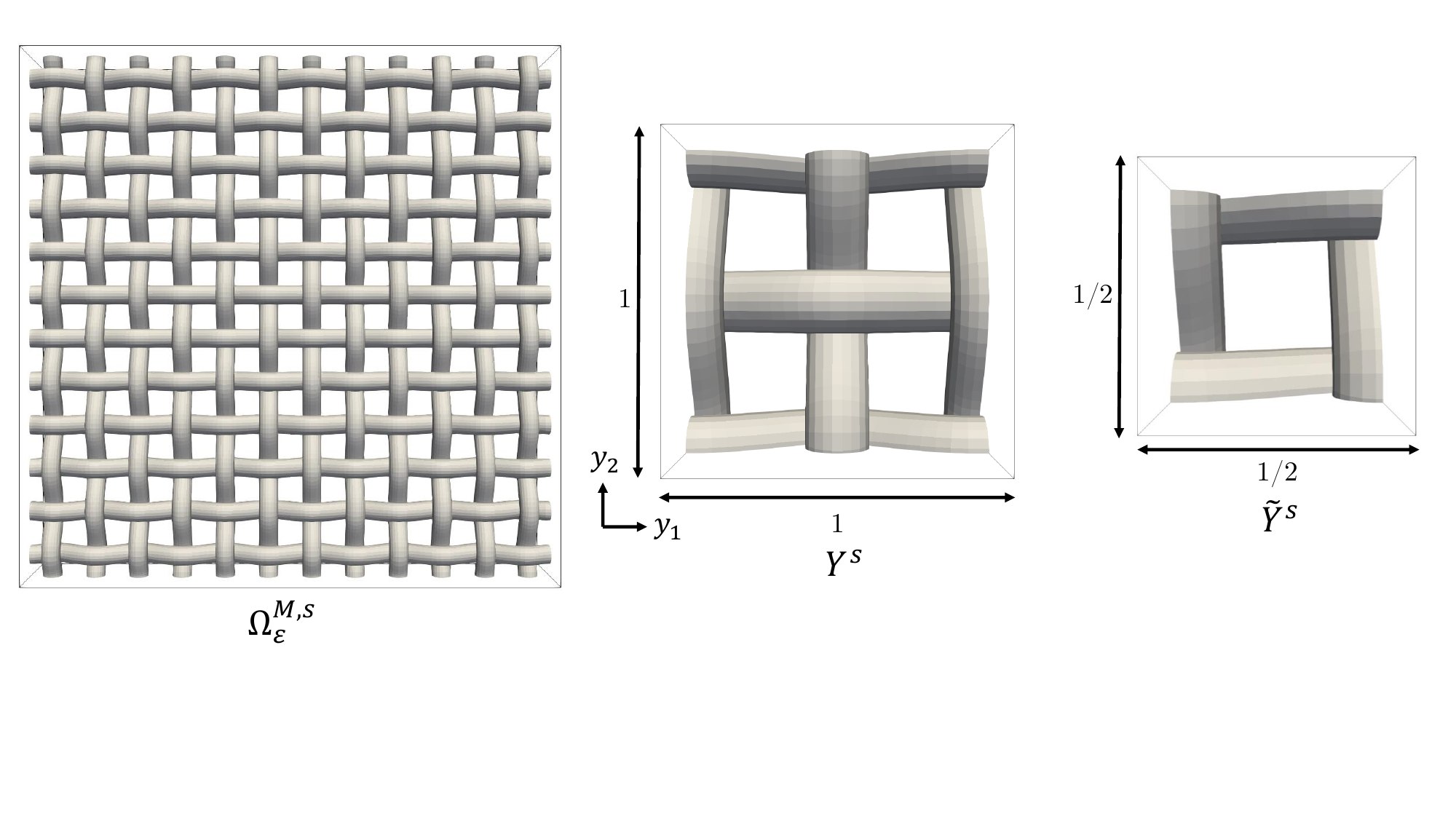}
		\caption[Example of symmetric filter]{Illustration of the structure domain $\O_\e^{M,s}$ (left), solid part of unit cell $Y^s$ (center) and quarter of the unit cell $\tilde{Y}^s$ (right) for a plain woven filter.}
		\label{fig:symmetryExample}
	\end{figure}
\end{example}

For the 1D beam formulation, it is assumed that each Lipschitz domain $\O$ in $Y^s$ can be described as a curved rod of length $L$ with constant cross-section of characteristic size $r>0$. That is, there exists a smooth curve 
\begin{equation*}
	\bm{\gamma}\colon [0,L]\to \R^3,
	\quad s_1\mapsto \bm{\gamma}(s_1),
	\quad \int_0^L \vert \bm{\gamma}^\prime (s_1)\vert \dx[s_1]= L
\end{equation*} 
parameterized by its arc-length
with well-defined Frenet-Serret frame
\begin{equation*}
	\bt(s_1) = \bm{\gamma}^\prime (s_1), \quad
	\bn(s_1) = \frac{\bt^\prime(s_1)}{\vert \bt^\prime(s_1)\vert}, \quad
	\bb(s_1) = \bt(s_1) \times \bn(s_1)
\end{equation*}
such that
\begin{equation*}
	\O = \{\bPhi(s)=\bm{\gamma}(s_1) + s_2\bn(s_1) + s_3\bb(s_1) : s=(s_1,s_2,s_3)\in (0,L)\times \omega_r\},
\end{equation*}
where $\omega_r = r \omega\subset \R^2$ is a Lipschitz domain centered around $\bm{0}$.

Under the above assumption, the dimension reduction approach from  \cite[Sec. 3]{griso:elementaryDisplacements} allows the restriction of displacement fields on $\O$ to a so-called elementary displacement along the curve $\bm{\gamma}$, \ie the centerline of the yarn. 
\begin{definition}
	Let $\bu \in L^1(\O)^3$ be given, which (with slight abuse of notation) is interpreted as a function of $s=(s_1,s_2,s_3)\in (0,L)\times \omega_r$ by considering $\bu\circ\bPhi$. Its elementary displacement is defined as
	\begin{equation}\label{eq:elementaryDisplacement}
		\bu_e(s) = \bU(s_1) + \bR(s_1) \times (s_2 \bn(s_1) + s_3 \bb(s_1)),
	\end{equation}
	where
	\begin{equation}\label{eq:expressions_elementaryDisplacement}
		\begin{aligned}
			\bU(s_1)
			&= \frac{1}{r^2 \vert\omega\vert} \int_{\omega_r} \bu(s_1,s_2,s_3) \dx[(s_2,s_3)], \\
			\bR(s_1) \cdot \bt(s_1)
			&= \frac{1}{(I_2+I_3)r^4} \int_{\omega_r} 
			\left( (s_2 \bn(s_1) + s_3 \bb(s_1)) \times \bu(s) \right)\cdot \bt(s_1) \dx[(s_2,s_3)], \\
			\bR(s_1) \cdot \bn(s_1)
			&= \frac{1}{I_3 r^4} \int_{\omega_r} 
			\left( (s_2 \bn(s_1) + s_3 \bb(s_1)) \times \bu(s) \right)\cdot \bn(s_1) \dx[(s_2,s_3)], \\
			\bR(s_1) \cdot \bb(s_1)
			&= \frac{1}{I_2 r^4} \int_{\omega_r} 
			\left( (s_2 \bn(s_1) + s_3 \bb(s_1)) \times \bu(s) \right)\cdot \bb(s_1) \dx[(s_2,s_3)]
		\end{aligned}
	\end{equation}
	and $I_k= \int_\omega s_k \dx[(s_2,s_3)], k=2,3$ are moments of area.
\end{definition}

Here, $\ba \times\bb$ denotes the standard cross-product in $\R^3$.
The representation \eqref{eq:elementaryDisplacement} can be understood as a displacement of the yarn centerline with an additional rotation of the yarn cross-section along the centerline. The remainder term
\begin{equation*}
	\bu_w =\bu - \bu_e
\end{equation*}
is commonly referred to as warping term and can be imagined as the deformation of the cross-section. In practical application, it is assumed to be small in comparison to the elementary displacement for slender structures. In fact, one has the following a priori estimate from Theorem 3.1 in \cite{griso:elementaryDisplacements}.
\begin{proposition}
	Let $\bu\in H^1(\Omega)^3$ and let $\bu_e,\bu_w$ denote its elementary displacement and the corresponding warping term, respectively. There exists $\overline{r}>0$, solely dependent on $\omega$ and $\bm{\gamma}$, such that there exist a uniform constant $\overline{c}>0$ with
	\begin{align*}
		\Vert \nabla \bu_w \Vert_{L^2(\Omega)}
		&\leq \overline{c} \Vert D(\bu) \Vert_{L^2(\Omega)}, \\
		\Vert \bu_w \Vert_{L^2(\O)} 
		&\leq \overline{c} r \Vert D(\bu) \Vert_{L^2(\Omega)}, \\		
		r \Vert \bR^\prime\Vert_{L^2((0,L))} + \Vert \bU^\prime - \bR\times \bt \Vert_{L^2((0,L))} 
		&\leq  \frac{\overline{c}}{r} \Vert D(\bu) \Vert_{L^2(\Omega)}
	\end{align*}
	for all $r<\overline{r}$. Here, $\bR^\prime,\bU^\prime$ denote the first-order derivative with respect to $s_1$.
\end{proposition}

By discretizing each yarn centerline $\bm{\gamma}$ by a finite sequence of piecewise linear segments, one attains a 1D frame structure with a sequence of nodes denoted by $(\bn_1,\dots,\bn_m)$.
The frame structure serves as a 1D FE mesh with the associated nodal DOF corresponding to the three centerline displacements $\bU$ and the three rotations $\bR$ from \cref{eq:expressions_elementaryDisplacement}, respectively. The interpolation method of choice are standard 1D beam elements in 3D space with 12 DOF per element.

The method is extended in \cite{vladimir:textileBeams} by the introduction of contact node pairs $(\bn_i,\bn_j)$ in-between two yarns, serving as an approximation of the Robin-type interface condition. For an extensive discussion of the assembly of the stiffness matrix and numerical analysis of the method for general linear elasticity problems, the reader is referred to \cite{vladimir:textileBeams}. 

For the incorporation of the generalized $Y$-periodic boundary conditions, an augmented master-slave approach is employed. It requires the evaluation of the FE interpolation of the perturbation functions $\bS_{ij}^{M,B}$ in each periodic node pair $(\bn_1,\bn_2)$ on the lateral boundaries of $Y^s$, respectively.

The resulting numerical solving routine of the cell problems and the computation of the stiffness tensor entries is implemented in the FiberFEM solver of the textile simulation software TexMath \cite{texmath}. As commonly encountered examples in real-world filtration application, the augmented cell solutions for a twill woven filter are presented in \Cref{fig:cellSolutions}. The remaining two solutions are given by rotational symmetry of the filter. For illustration purposes, the periodic unit was repeated five times in each in-plane direction.
\begin{figure}[htbp]
	\centering
	\includegraphics[width=0.8\textwidth]{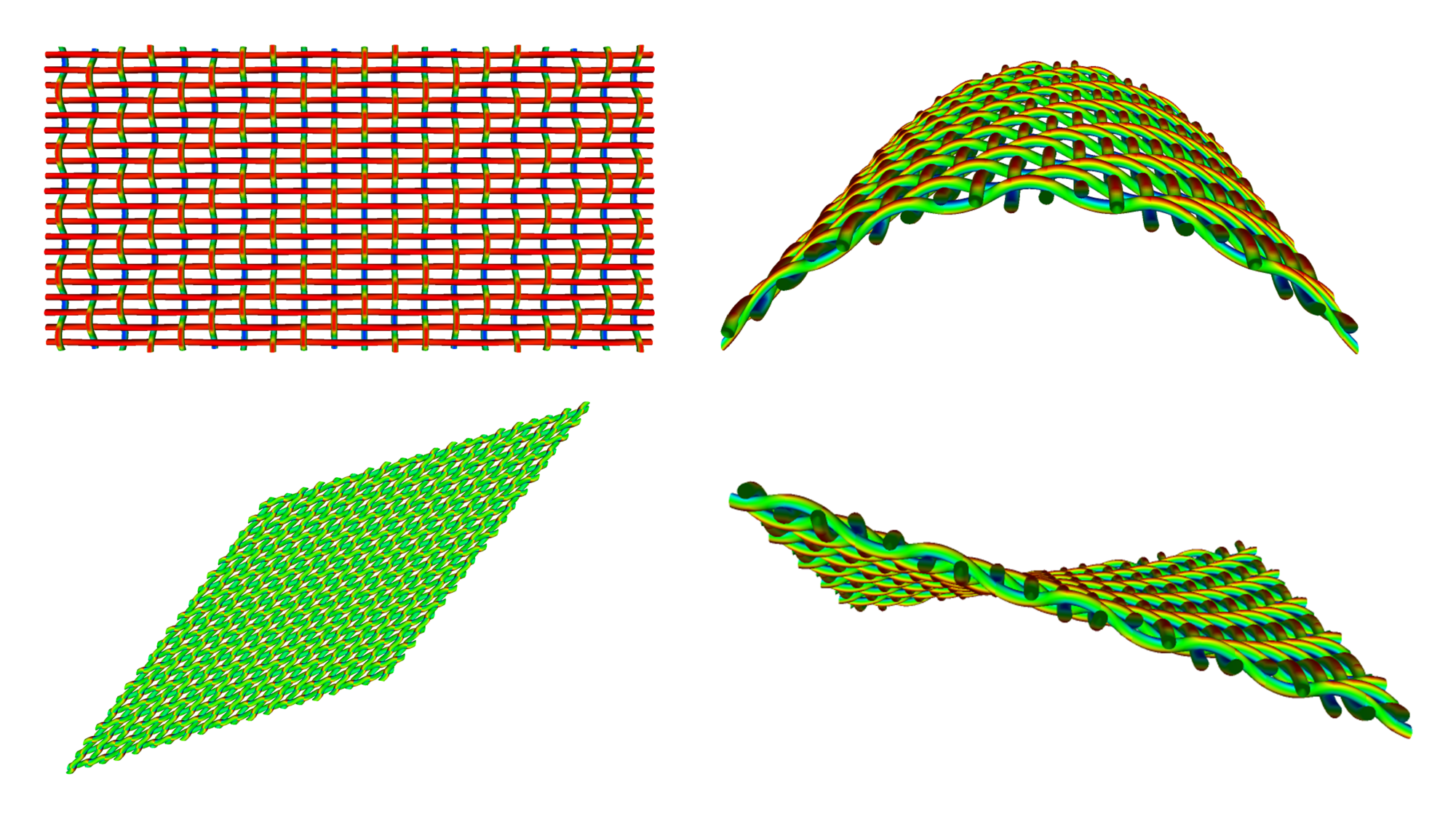}
	\caption{Augmented cell solutions $\bm{m}_{11}^M,\bm{m}_{11}^B$ (top) and $\bm{m}_{12}^M,\bm{m}_{12}^B$ (bottom) interpreted as displacement fields of a twill woven filter. Colors indicate local stresses.}
	\label{fig:cellSolutions}
\end{figure}

\section{Computation of the permeability}
\label{sec:permeability}

The computation of the permeability tensor $\bK$ is performed by standard-means, utilizing a voxel discretization of the filter structure attained from TexMath. The common approach is actually not to solve the provided cell problems \cref{eq:cellProblemsFluid}, but to perform an approximation procedure based on Darcy's law. The computational effort for both approaches is expected to be comparable.

The methodology starts by performing three (stationary) Stokes flow simulations in a fully resolved reference cell with a prescribed constant pressure drop $\llbracket p_i\rrbracket \in\R,i=1,2,3$ along the main axes, respectively. For the remaining boundaries, periodic boundary conditions are applied.

In a next step, from the attained solutions $(\bv_i,p_i)$, the average velocities $\hat{\bv}_i \in\R^3, i=1,2,3$ are computed. Afterwards, by approximating the pressure gradient by the finite difference
\begin{equation*}
	\nabla p_i \approx \frac{\llbracket p_i\rrbracket}{L_i} \be_i \in \R^3
\end{equation*}
with $L_i$ denoting the respective physical length of the structure in $x_i$-direction, one can approximate $\bK$ by the solution of the system of linear equations
\begin{equation*}
	\hat{\bv}_i = -\frac{\llbracket p_i\rrbracket}{L_i \mu} \bK \be_i, \quad i=1,2,3
\end{equation*}
under the assumption that Darcy's law applies.
In the considered case, $L_3=\delta$ is the characteristic thickness of the filter and the remaining lengths are given by the period $\e$.

By linearity, the computed tensor $\bK$ is independent of the choice of $\llbracket p_i\rrbracket$, as well as $\mu$. Exemplary flow solutions $(\bv_i,p_i)$ for a twill woven filter with $\llbracket p_i\rrbracket = \SI{1}{\pascal}$ and $\mu=\SI{1e-3}{\pascal\second}$ are presented in Figure \ref{fig:examplePermeabilityFlow}. They are attained utilizing the LIR-Stokes solver of the software GeoDict.

\begin{figure}[htbp]
	\centering
	\includegraphics[width=0.8\textwidth]{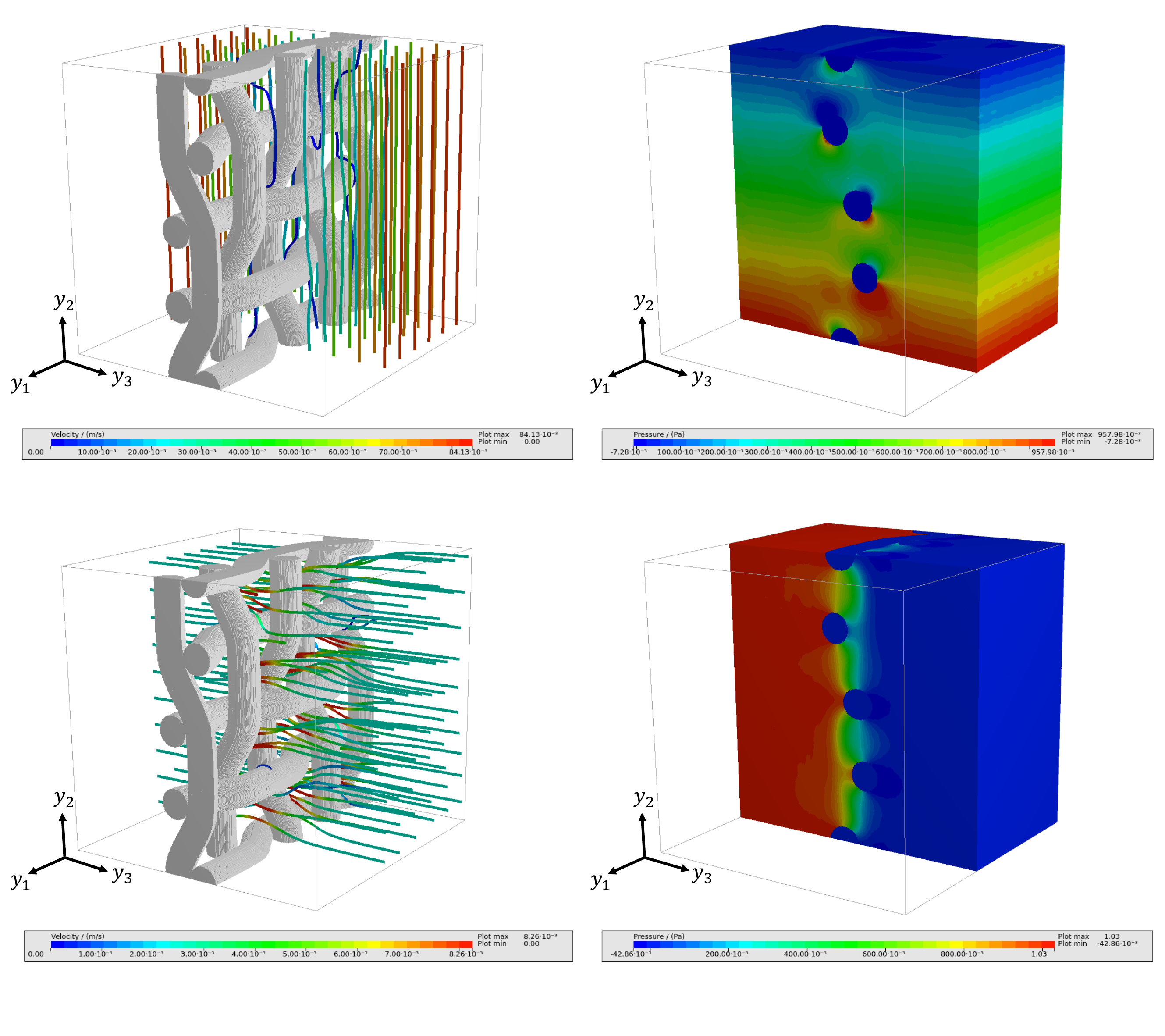}
	\caption{Flow solutions $(\bv_2,p_2)$ (top) and $(\bv_3,p_3)$ (bottom) for a twill woven filter. The remaining flow solution is of similar nature due to symmetry of the structure.}
	\label{fig:examplePermeabilityFlow}
\end{figure}

\section{Monolithic FSI solver}
\label{sec:FSISolver}

The numerical method to solve the FSI system \cref{eq:system_macro} is split into two phases. In the first phase, the macroscopic model parameters $\aahom,\bbhom,\cchom$, as well as $\bK$, are computed utilizing the microscopic routines from the previous two sections. Afterwards, a monolithic finite element discretization of \cref{eq:system_macro} is employed, that is fluid and structure equations are solved as a single discrete system.

For the derivation of the FE system, the auxiliary variable $w_3=\partial_t u_3$ for the plate's normal velocity is introduced. Furthermore, the required space-variable dependent function space is denoted by
\begin{equation*}
	\cY = \cV \times L^2(\Opm) \times H_0(\Sigma)^2 \times H_0^2(\Sigma) \times L^2(\Sigma).
\end{equation*}
With this notation, Rothe's method is employed for the semi-discretization of system \cref{eq:system_macro} in time.
For this purpose, let $\bdt[n+1]=t^{n+1}-t^n$ with discrete time steps $0=t^0<t^1<\dots<t^N=T$ for some $N\geq 1$. The approximation of partial derivatives in time is performed via backwards difference quotients
\begin{displaymath}
	\partial_t u(t^{n+1}) \approx \frac{u^{n+1} - u^{n}}{\bdt[n+1]}
\end{displaymath}
for purely space dependent functions $\by^n=(\bv^{n},p^n,\baru^{n},u_3^{n},w_3^{n})\in\cY$, approximating the solution at time $t^n$. For $n=0$, the approximation is given by the initial data.

For the right-hand side functions, the semi-discretization in time reads
\begin{equation*}
	\bm{f}^{n+1} = \frac{1}{\bdt[n+1]}\int_{t^n}^{t^{n+1}} \bm{f}(\tau)\dx[\tau], \quad
	g_3^{n+1} = \frac{1}{\bdt[n+1]}\int_{t^n}^{t^{n+1}} g_3(\tau)\dx[\tau].
\end{equation*}

Standard procedure delivers the variational formulation.

\begin{lemma}
	The variational formulation of the semi-discretized system \cref{eq:system_macro} for step $n+1$ consists of finding $\by^{n+1} = (\bv^{n+1},p^{n+1},\baru^{n+1},u_3^{n+1},w_3^{n+1})\in\mathcal{Y}$ such that
	\begin{equation}\label{eq:numerics_semiDiscrete}
		\begin{aligned}
			&\frac{\rho_f}{\bdt[n+1]} (\bv^{n+1}, \bV)_{\Opm}
			+ 2\mu (D(\bv^{n+1}), D(\bV))_{\Opm} 
			- (p^{n+1}, \nabla\cdot \bV)_{\Opm} \\
			&\quad
			+ (\hk^{-1} \bv^{n+1}, \bV)_\Sigma 
			- \frac{1}{\bdt[n+1]} ( \hk^{-1} u_3^{n+1}\be_3, \bV)_\Sigma \\
			&= (\bm{f}^{n+1}, \bV)_{\Opm} 
			+ \frac{\rho_f}{\bdt[n+1]} (\bv^n, \bV)_{\Opm}
			- \frac{1}{\bdt[n+1]} (\hk^{-1} u_3^n\be_3, \bV)_\Sigma, \\
			&-(\nabla\cdot \bv^{n+1}, P)_{\Opm}
			= 0, \\
			&\frac{\hrho}{\bdt[n+1]}(w_3^{n+1},U_3)_\Sigma
			+ \ahom( (\baru^{n+1},u_3^{n+1}), (\barU,U_3)) 
			- (\hk^{-1} \bv^{n+1}, U_3\be_3)_\Sigma \\
			&\quad
			+ \frac{1}{\bdt[n+1]} (\hk^{-1} u_3^{n+1}\be_3, U_3\be_3)_\Sigma \\
			&= 
			(g_3^{n+1}, U_3)_\Sigma
			+ \frac{\hrho}{\bdt[n+1]}(w_3^{n},U_3)_\Sigma
			+ \frac{1}{\bdt[n+1]} (\hk^{-1} u_3^n \be_3, U_3 \be_3)_\Sigma, \\
			&\hrho(w_3^{n+1},W_3)_\Sigma 
			- \frac{\hrho}{\bdt[n+1]}(u_3^{n+1},W_3)_\Sigma 
			= 
			-\frac{\hrho}{\bdt[n+1]} (u_3^n, W_3)_\Sigma
		\end{aligned}
	\end{equation}
	for all $(\bV,P,\barU,U_3,W_3)\in \mathcal{Y}$.
\end{lemma}

Utilizing \Cref{proposition:stiffnessTensorsCoercive} and \Cref{proposition:cellProblemsFluid}, the well-posedness of \cref{eq:numerics_semiDiscrete} follows with the classical LBB theorem.

\begin{theorem}
	For all $n=0,\dots,N-1$, the semi-discrete system \cref{eq:numerics_semiDiscrete} has a unique solution $\by^{n+1}$.
\end{theorem}

\begin{proof}
	For easier notation, it is assumed that all arising scalar constants, apart from $[\Delta t^{n+1}]$, are equal to $1$. Furthermore, the superscript $n+1$ is omitted whenever it is clear from context.
	
	The proof is performed by induction. Let $n$ be given.
	With the assumptions above and after introducing the scaled test functions $\bdt \bV$ and $\bdt P$, system \eqref{eq:numerics_semiDiscrete} can be abstracted to
	\begin{align*}
		a(\bm{\phi},\bPhi) + b(\bPhi,p) 
		&= L[\bm{\Phi}] \\
		b(\bm{\phi},P) &= 0,
	\end{align*}
	where $\bm{\phi} = (\bv,\baru,u_3,w_3), \bPhi = (\bV,\barU,U_3,W_3)$, $L$ is the bounded linear functional 
	\begin{align*}
		L[\bPhi] &= 
		(\bm{f}^{n+1}, \bV)_{\Opm} 
		+ \frac{\rho_f}{\bdt} (\bv^n, \bV)_{\Opm}
		- \frac{1}{\bdt} (\hk^{-1} u_3^n \be_3, \bV)_\Sigma \\
		&\quad
		+ (g_3^{n+1}, U_3)_\Sigma
		+ \frac{\hrho}{\bdt}(w_3^{n},U_3)_\Sigma
		+ \frac{1}{\bdt} (\hk^{-1} u_3^n\be_3, U_3\be_3)_\Sigma \\
		&\quad
		- \frac{\hrho}{\bdt} (u_3^n, W_3)_\Sigma
	\end{align*}
	with solutions from previous time steps treated as given data and
	\begin{align*}
		a(\bm{\phi},\bm{\Phi}) 
		&=
		(\bv, \bV)_{\Opm} 
		+ \bdt (D(\bv),D(\bV))_{\Opm} 
		+ \bdt (\hk^{-1} \bv,\bV)_\Sigma \\
		& \quad
		- (\hk^{-1} u_3 \be_3, \bV)_\Sigma
		+ \bdt^{-1} (w_3, U_3)_\Sigma
		+ \ahom((\baru,u_3),(\barU,U_3))  \\
		& \quad
		- (\hk^{-1} \bv, U_3\be_3)_\Sigma 
		+ \bdt^{-1} (\hk^{-1} u_3\be_3, U_3\be_3)_\Sigma 
		+ (w_3, W_3)_\Sigma \\
		&\quad
		- \bdt^{-1}(u_3, W_3)_\Sigma, \\
		b(\bm{\phi},P) 
		&=
		-\bdt (\nabla\cdot \bv, P)_{\Opm}.
	\end{align*}
	
	Similar to classical Stokes theory, the coercivity of the bilinear form $a$ can be ensured on the entirety of $\mathcal{V}\times H_0^1(\Sigma)^2 \times H^2_0(\Sigma) \times L^2(\Sigma)$ with \Cref{proposition:stiffnessTensorsCoercive} and \Cref{proposition:cellProblemsFluid}, since
	\begin{align*}
		a(\bm{\phi},\bm{\phi}) &=
		\Vert \bv \Vert_{L^2(\Opm)} ^2
		+ \bdt\Vert D(\bv)\Vert^2_{L^2(\Opm)} \\
		&\quad
		+ \Vert \hk^{-\frac{1}{2}}( \bdt^{\frac{1}{2}} v_3 - \bdt^{-\frac{1}{2}} u_3\be_3 ) \Vert^2_{L^2(\Sigma)} \\
		&\quad
		+ \Vert (\baru, u_3) \Vert^2_\text{hom}
		+ \Vert w_3\Vert_{L^2(\Sigma)}^2,
	\end{align*}
	where $\hk^{-\frac{1}{2}}$ denotes the unique square root of $\hk^{-1}$.
	In particular, $a$ is coercive on the kernel of $b$.
	
	Again from classical Stokes theory, one can further deduce that independent of the choice of $(\baru,u_3,w_3)$, there exists a constant $\underline{c}>0$, such that for all $p\vert_{\O^\pm}$ with $p\in \mathcal{P}$ the LBB condition
	\begin{equation}\label{eq:LBBconditionExistenceProof}
		\sup_{\substack{\bv\in\mathcal{V} \\ \bv\vert_{\O^\pm} \neq \bm{0}}}
		\frac{(\nabla\cdot \bv\vert_{\O^\pm}, p\vert_{\O^\pm})_{\O^\pm}}{\Vert \bv\Vert_{H^1(\O^\pm)}}  \geq \underline{c} \Vert p\Vert_{L^2(\O^\pm)}
	\end{equation}
	for each subdomain $\O^\pm$ is fulfilled. 
	The statement then follows by inductive application of the LBB theorem.
\end{proof}

With the established existence of solutions in the semi-dicrete setting, the system \cref{eq:numerics_semiDiscrete} is further discretized with respect to the space variable. For this purpose, conforming FE are chosen, \ie one chooses finite dimensional approximation spaces
\begin{gather*}
	\cV^h \subset \cV, \quad
	\cP^h\subset L^2(\Opm), \quad
	\bar{\cU}^h\subset H_0^1(\Sigma)^2, \quad
	\cU_3^h \subset H_0^2(\Sigma), \quad
	\cW_3^h\subset L^2(\Sigma)
\end{gather*}
and sets $\cY^h=\cV^h \times \cP^h \times \bar{\cU}^h \times \cU_3^h \times \cW_3^h$. Here and in the following, $h$ denotes a characteristic element size for spatial decomposition of the computational domain $\O$. \\

In what follows, let 
\begin{equation*}
	(\{\bV^h_k\},\{P^h_k\},\{\barU^h_k\},\{{U^h_3}_k\}, \{{W^h_3}_k\})
\end{equation*}
form a basis of $\cY^h$ and
let further 
\begin{equation*}
	\by^{n,h} = (\bv^{n,h}, p^{n,h}, \baru^{n,h}, u_3^{n,h}, w_3^{n,h})^T \in \cY^h
\end{equation*} 
be an approximation of $\by^n\in \cY$. The semi-discrete solution variable $\by^{n,h}$ is associated with its DOF vector, also denoted by $\by^{n,h}$.

With this notation, $\by^{n+1,h}$ is the solution to 
\begin{equation}\label{eq:macro_fullyDiscrete}
	\left(\frac{1}{\bdt[n+1]}\bS_1 + \bS_2\right)\by^{n+1,h} = \frac{1}{\bdt[n+1]}\bS_1 \by^{n,h} + \bL(t^{n+1})
\end{equation}
with the system matrices
\begin{align*}
	\bS_1 &\coloneqq 
	\begin{pmatrix}
		\bM_{VV} & 0 & 0 & -\bR_{VU} & 0 \\
		0 & 0 & 0 & 0 & 0 \\
		0 & 0 & 0 & 0 & 0 \\
		0 & 0 & 0 & \bR_{UU} & \bM_{UW} \\
		0 & 0 & 0 & -\bM_{UW}^T & 0
	\end{pmatrix}, \\
	\bS_2 &\coloneqq
	\begin{pmatrix}
		\bA + \bR_{VV} & -\bB^T& 0 & 0 & 0 \\
		-\bB  & 0 & 0 & 0 & 0 \\
		0 & 0 & \bP_A & \bP_{B_1} & 0 \\
		-\bR_{VU}^T & 0 & \bP_{B_2} & \bP_C & 0 \\
		0 & 0 & 0 & 0 & \bM_{WW}
	\end{pmatrix}
\end{align*}
consisting of the constant block matrices
\begin{equation*}
	\begin{array}{l l l l}
		\bM_{VV} 
		&= \left( \rho_f( \bV^h_k, \bV^h_l)_{\Opm} \right)_{kl}, 
		&\bM_{UW} 
		&= \left( \hrho( {U_3}^h_k, {W_3}^h_l)_\Sigma \right)_{kl},\\
		\bM_{WW} 
		&= \left( \hrho( {W_3}^h_k, {W_3}^h_l)_\Sigma \right)_{kl},
		&\bR_{VV} 
		&= \left( (\hk^{-1} \bV^h_k  , \bV^h_l)_\Sigma \right)_{kl}, \\
		\bR_{VU} 
		&= \left( ( \hk^{-1} \bV^h_k , {U_3}^h_l\be_3)_\Sigma \right)_{kl},
		&\bR_{UU} 
		&= \left( ( \hk^{-1} {U_3}^h_k \be_3, {U_3}^h_l \be_3)_\Sigma \right)_{kl}, \\
		\bA 
		&= \left( 2\mu( D(\bV^h_k), D(\bV^h_l))_{\Opm} \right)_{kl}, 
		&\bB 
		&= \left( (P^h_k ,\nabla\cdot \bV^h_l)_{\Opm} \right)_{kl}, \\
		\bP_A 
		&= \left( ( \aahom D_\barx(\barU^h_l), D_\barx(\barU^h_k))_\Sigma \right)_{kl}, \\
		\bP_C 
		&= \left( ( \cchom \nabla^2_\barx({U^h_3}_l), \nabla^2_\barx({U^h_3}_k))_\Sigma \right)_{kl}, \\
		\bP_{B_1} 
		&= \left( ( \bbhom \nabla^2_\barx({U^h_3}_l), D_\barx(\barU^h_k))_\Sigma \right)_{kl}, \\
		\bP_{B_2} 
		&= \left( ( \bbhom D_\barx(\barU^h_l), \nabla^2_\barx({U^h_3}_k))_\Sigma \right)_{kl}
	\end{array}
\end{equation*}
and time dependent right-hand side $\bL(t)=(\bF(t),0,0,\bG_3(t),0)^T$ with blocks
\begin{equation*}
	\bF(t) = \left( ( \bm{f}(t), \bV^h_k)_{\Opm} \right)_{k}, \quad
	\bG_3(t) = \left( -( g_3(t), {U^h_3}_k)_\Sigma \right)_{k}.
\end{equation*}
Note that in general one has $\bP_{B_1}\neq \bP_{B_2}^T$. \\

For the stationary case, the fully discrete formulation consists of the two linear systems
\begin{equation*}
	\begin{pmatrix}
		\bA + \bR_{VV} & -\bB^T  \\
		-\bB  & 0
	\end{pmatrix} 
	\begin{pmatrix}
		\bv^h \\
		p^h
	\end{pmatrix}
	=
	\begin{pmatrix}
		\bm{F} \\
		0
	\end{pmatrix},
	\quad
	\begin{pmatrix}
		\bP_A & \bP_{B_1}  \\
		\bP_{B_2} & \bP_C
	\end{pmatrix} 
	\begin{pmatrix}
		\baru^h \\
		u_3^h
	\end{pmatrix}
	=
	\begin{pmatrix}
		0 \\
		\bG_3 + \bR_{VU}^T \bv^h
	\end{pmatrix},
\end{equation*}
which can be solved in sequential order. \\

For the choice of specific finite element spaces, a spatial decomposition of $\O$ using a regular hexahedral mesh is proposed. The mesh is chosen as $\Sigma$-conforming in the sense that its restriction to the interface $\Sigma$ is a quadrilateral 2D mesh given by the element facets. A reformulation with a tetrahedral decomposition is straightforward.

In what follows, the hexahedral elements are denoted by $T\in\cT^h$, while the facets of $\Sigma$ are denoted by $F\in\cF^h$. \\

For the velocity variables, the classical $\cQ_2/\cQ_1$ Taylor-Hood pairing is chosen. Since $\bv$ is continuous on $\Sigma$, while $p$ has a jump discontinuity on $\Sigma$, the respective finite element spaces read
\begin{displaymath}
	\begin{aligned}
		\cV^h &= 
		\{\bv^h \in C^0(\O)^3 : v_i^h \vert_T \in \cQ_{k+1} \text{ for all }T\in\cT^h, i=1,2,3\} \cap \cV, \\
		\cP^h &= 
		\{p^h : p^h\vert_{\O^\pm} \in C^0(\O^\pm), p^h \vert_T \in \cQ_k \text{ for all }T\in\cT^h\}.	
	\end{aligned}
\end{displaymath}
The resulting pressure mesh has a fissure on $\Sigma$, with each mesh node on $\Sigma$ being associated with two pressure DOF, respectively. The authors in \cite{bloodFlowPorousInterface} additionally performed comparative studies for a similar stationary Stokes-Stokes problem with globally continuous pressure space. Unsurprisingly, this choice leads to poor results unless the discretization size is sufficiently small around $\Sigma$. \\

For the plate's in-plane displacement, $\cQ_1$ interpolation in 2D is employed. The $H^2$-conformity of the deflection $u_3$ requires continuous first-order derivatives of the FE across edges in the mesh, \ie $C^1$-elements. For quadrilateral meshes, the employment of Bogner-Fox-Schmit (BFS) elements (see \cite{BFS,Ciarlet:FEM}) is proposed, which are bicubic polynomials that are comparatively easy to self-implement. We note, that for unstable lattice structures (see \cite{GKOS2}), the mixed derivatives are unknown. Recent work \cite{riccardo:latticeStructures} offers even more effective $Q_3$ or be-cubic interpolation, avoiding mixed derivatives.\\ The corresponding FE spaces are
\begin{displaymath}
	\begin{aligned}
		\bar{\cU}^h &= 
		\{\baru^h \in C^0(\Sigma)^2 : \bar{u}^h_i\vert_F \in \cQ_1 \text{ for all }F\in\cF^h,i=1,2 \}\cap \bar{\cU}, \\
		\cU_3^h &= 
		\{u_3^h\in C^1(\Sigma) : u_3^h\vert_F \in \cQ_3 \text{ for all }F\in \cF^h\} \cap \cU_3.
	\end{aligned}
\end{displaymath}

A standard basis of the BFS elements in $\R^n$ is attained from tensor products of classical 1D Hermite splines. On the unit interval $[0,1]$, the latter read
\begin{alignat*}{2}
	&\hat{H}_{00}(x) = (2x + 1) (x-1)^2, \quad && \hat{H}_{10}(x) = x(x-1)^2, \\
	&\hat{H}_{01}(x) = x^2(3-2x), \quad && \hat{H}_{11}(x) = x^2(x-1),
\end{alignat*}
which generalizes to arbitrary intervals $[x^0,x^1]$ with length $L=x^1-x^0$ by the affine change of variables
\begin{alignat*}{2}
	&H_{00}(x) = \hat{H}_{00}\left(\frac{x-x^0}{L}\right), 
	\quad && H_{10}(x) = L\hat{H}_{10}\left(\frac{x-x^0}{L}\right), \\
	&H_{01}(x) = \hat{H}_{01}\left(\frac{x-x^0}{L}\right), 
	\quad && H_{11}(x) = L\hat{H}_{11}\left(\frac{x-x^0}{L}\right).
\end{alignat*}
A sketch of the 1D splines on the unit interval is shown in Figure \ref{fig:hermiteShapeFunctions}. The corresponding 1D FE are commonly referred to as \textit{Hermite} elements, see \cite{Ciarlet:Hermite}.

\graphicspath{{Images/FEM/}}

\begin{figure}[th]
	\centering
	\includegraphics[width=0.35\textwidth]{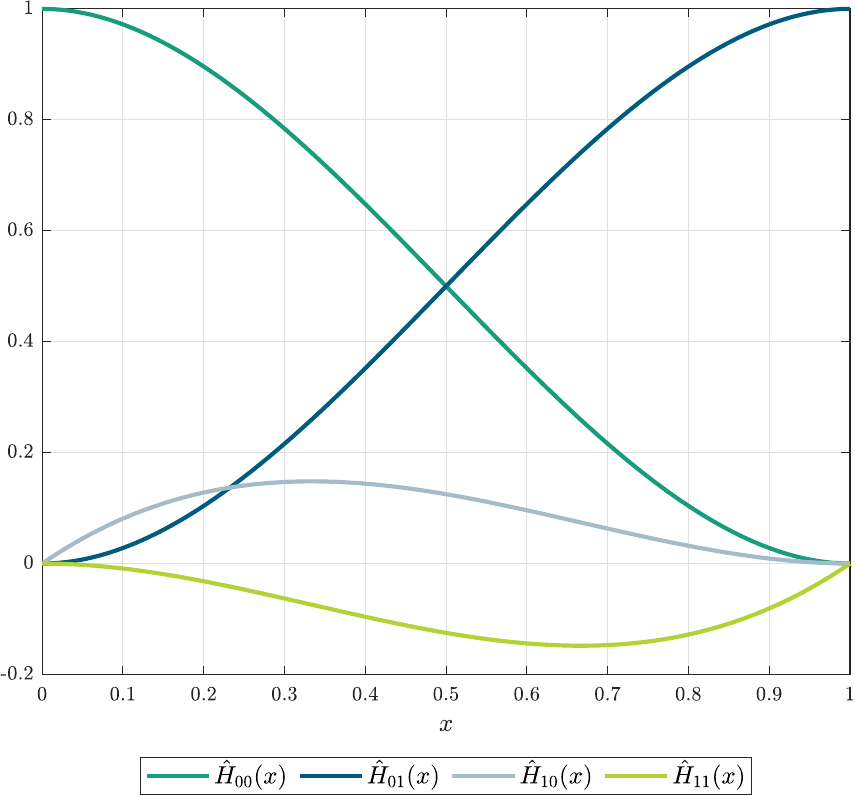}
	\caption[Hermite shape functions]{Hermite splines on unit interval.}
	\label{fig:hermiteShapeFunctions}
\end{figure}

\newcommand{\balpha}{\bm{\alpha}}
\newcommand{\bbeta}{\bm{\beta}}

The BFS basis polynomials are attained by computing tensor products of the Hermite splines. Let $[x^0_1,x^1_1]\times \dots \times [x^0_n,x^1_n]$ denote an arbitrary cuboid in $\R^n$ with edge lengths $L_i = x^1_i - x^0_i$. We define the $4^n$ basis polynomials as
\begin{equation*}
B_{\balpha ,\bbeta }(x) \coloneqq 
\prod_{i=1}^n 
L_i^{\alpha_i} \hat{H}_{\alpha_i \beta_i}
\left(\frac{x_i-x^0_i}{L_i}\right), 
\quad \balpha,\bbeta\in \{0,1\}^n
\end{equation*}
which for the unit cube $[0,1]^n$ results in the reference functions
\begin{equation*}
\hat{B}_{\balpha ,\bbeta}(x_1,x_2) \coloneqq 
\prod_{i=1}^n 
\hat{H}_{\alpha_i \beta_i}
(x_i), 
\quad \balpha,\bbeta\in \{0,1\}^n.
\end{equation*}

We can derive that the interpolant of $w\in C^1([x^0_1,x^1_1]\times \dots \times [x^0_n,x^1_n])$ by the BFS polynomials reads
\begin{equation*}
\Pi_{\text{BFS}}[w](x) = \sum_{\balpha,\bbeta\in\{0,1\}^n}
\partial^{\alpha_{1}}_{x_1}\cdots\partial^{\alpha_n}_{x_n} 
w(x_1^{\beta_1},\dots,x_n^{\beta_n}) B_{\balpha,\bbeta}(x).
\end{equation*}
Therefore, we can associate the $4^n$ nodal DOF $\partial^{\alpha_{1}}_{x_1}\cdots\partial^{\alpha_n}_{x_n} 
w(x_1^{\beta_1},\dots,x_n^{\beta_n})$ per element. 

For the specific choice of $n=2$, for each mesh node on $\Sigma$, we attain the deflections value, the value of its two first-order derivatives as well as the value of the mixed second-order derivative. In Figure \ref{fig:bfsShapeFunctions}, four of the sixteen derived shape functions for $n=2$ are plotted that are associated with the point $(1,0)$. The remaining functions are of similar nature.

\begin{figure}[H]
	\centering
	\includegraphics[width=0.35\textwidth]{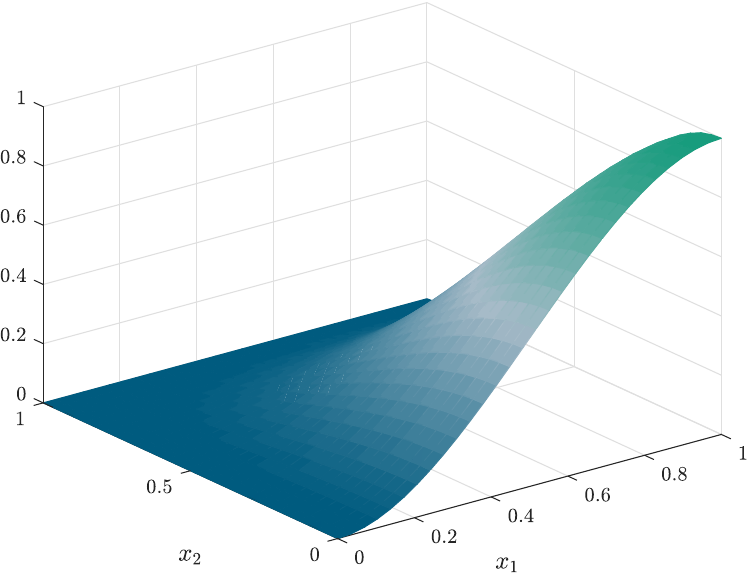}\hspace{0.4cm}
	\includegraphics[width=0.35\textwidth]{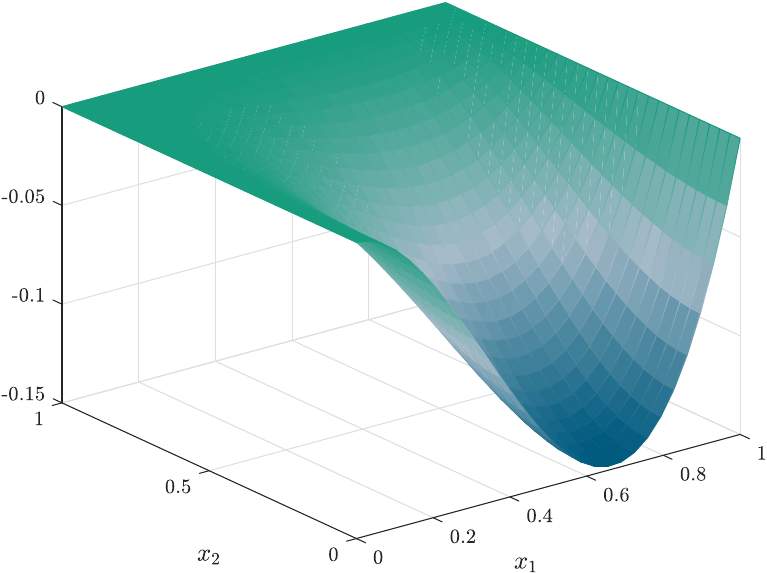}
	\vspace{0.5cm}
	
	\includegraphics[width=0.35\textwidth]{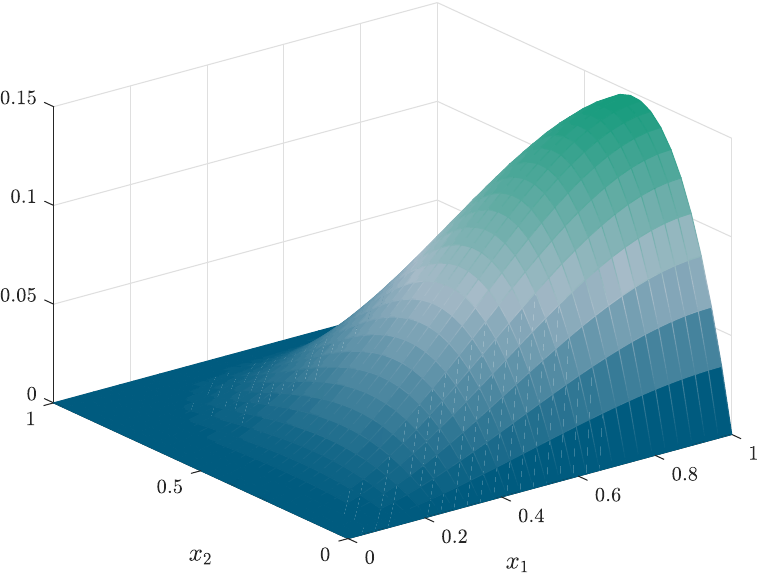}\hspace{0.4cm}
	\includegraphics[width=0.35\textwidth]{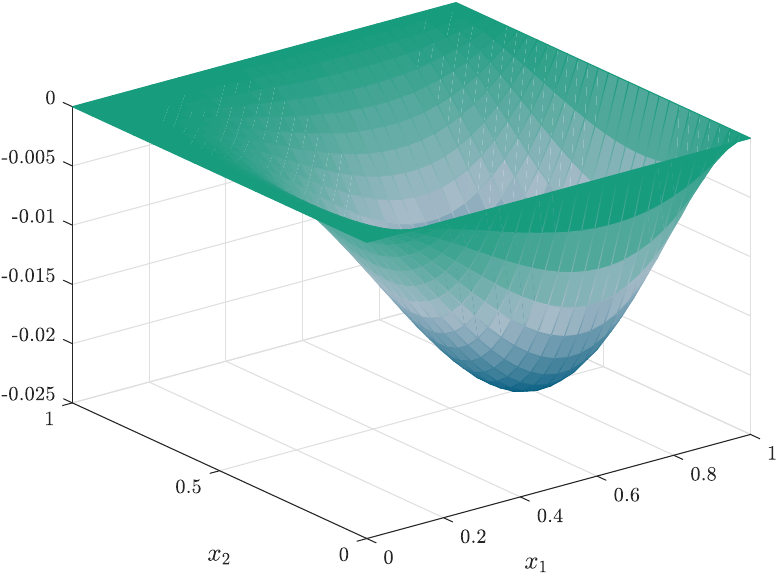}
	\caption[Bogner-Fox-Schmit shape functions]{The four BFS shape functions $\hat{B}_{(0,0),(1,0)}, \hat{B}_{(1,0),(1,0)}, \hat{B}_{(0,1),(1,0)}, \hat{B}_{(1,1),(1,0)}$ on the unit square.}
	\label{fig:bfsShapeFunctions}
\end{figure}

The resulting amount of DOF per respective element for the chosen spatial discretization is summarized in Table \ref{table:summaryFE}.
\begin{table}[H]
	\centering
	\begin{tabular}{l|l|l|l|l}
		Function & $\bv^h$ & $p^h$ & $\baru^h$ & $u_3^h$ \\ \hline
		FE type & $\cQ_2$ & $\cQ_1$ & $\cQ_1$ & BFS \\ \hline
		DOF per element & 27 & 8 & 4 & 16
	\end{tabular}
	\caption[Summary of employed FE]{Summary of employed FE.}
	\label{table:summaryFE}
\end{table}

We expect the following errors for the employed spatial FE under idealized time-stepping, see \cite{GiraultVivette:FEStokes, Ciarlet:FEM}. 
\begin{proposition}\label{proposition:errorEstimates}
	Let $n\in\{1,\dots,N\}$ be given and assume that the right-hand side of \eqref{eq:numerics_semiDiscrete} is given by the exact solution.
	Then, for sufficiently regular solutions, the chosen spatial interpolation methods provide the a priori error estimates
	\begin{align*}
	&\Vert \bv^n - \bv^{n,h} \Vert_{H^1(\Opm)} + 
	\Vert p^n - p^{n,h}\Vert_{L^2(\Opm)} \\
	&\quad\quad\quad\quad\leq 
	\overline{c}_1 h^2 \left(\vert \bv^n \vert_{H^3(\Opm)} + \vert p^n\vert_{H^2(\Opm)} \right), \\
	&\Vert \bv^n - \bv^{n,h} \Vert_{L^2(\Opm)} \\
	&\quad\quad\quad\quad\leq 
	\overline{c}_1 h^3 \left(\vert \bv^n \vert_{H^3(\Opm)} + \vert p^n\vert_{H^2(\Opm)} \right)
	\end{align*}
	for the fluid variables for some constants $\overline{c}_1>0$ independent of $h$.
	
	Further, the error for the displacement variables satisfies the elliptic estimates
	\begin{align*}
	\Vert \baru^n - \baru^{n,h} \Vert_{H^{1-k}(\Sigma)} 
	&\leq \overline{c}_2 h^{1+k}\vert \baru^n \vert_{H^2(\Sigma)}, \\
	\Vert u_3^n - u_3^{n,h} \Vert_{H^{2-m}(\Sigma)} 
	&\leq \overline{c}_2 h^{2+m}\vert u_3^n \vert_{H^{4}(\Sigma)},
	\end{align*}
	for $k=0,1$ and $m=0,1,2$ and a constant $\overline{c}_2>0$ independent of $h$. Here $\Vert\cdot\Vert_{H^0(\Sigma)}$ corresponds to the $L^2$-norm.
\end{proposition}

\section{Simulation results}
\label{sec:simulationResults}

\subsection{Qualitative description of stiffness tensors}
\label{section:qualitativeDescriptionTensors}
\graphicspath{{Images/Homogenization/}}

In the following section, we qualitatively describe the influence of entries in the homogenized stiffness tensors on the overall behavior of the homogenized textile under different loading scenarios. The discussion enables the quantitative analysis of the entries in the subsequent section.

\begin{remark}
	With the knowledge about symmetry of the homogenized stiffness tensors $\aahom,\cchom$, we deduce that there are at most six independent entries per tensor,
	which we represent in a symmetric $3\times 3$ matrix of the form
	\begin{equation*}
	\aahom = 
	\begin{pmatrix}
	\ahom_{1111} & \ahom_{1122} & \ahom_{1112} \\
	\ast & \ahom_{2222} & \ahom_{2212} \\
	\ast & \ast & \ahom_{1212}
	\end{pmatrix},
	\quad
	\cchom = 
	\begin{pmatrix}
	\chom_{1111} & \chom_{1122} & \chom_{1112} \\
	\ast & \chom_{2222} & \chom_{2212} \\
	\ast & \ast & \chom_{1212}
	\end{pmatrix}.
	\end{equation*}
	With the knowledge about the reduced symmetry of the coupling stiffness tensor, we write
	\begin{equation*}
	\bbhom = 
	\begin{pmatrix}
	\bhom_{1111} & \bhom_{1122} & \bhom_{1112} \\
	\bhom_{2211} & \bhom_{2222} & \bhom_{2212} \\
	\bhom_{1211} & \bhom_{1222} & \bhom_{1212}
	\end{pmatrix}.
	\end{equation*}
\end{remark}

We start our qualitative description with a result from \cite{Panasenko:Multi-ScaleModellingOfStructures} that describes the effective properties of orthotropic plates.
\begin{lemma}\label{lemma:isotropicPlate}
	Assume that the microscopic structure is given by an orthotropic 3D plate with Young's moduli $E_1,E_2$, Poisson's ratios $\nu_{12},\nu_{21}$, shear modulus $G$ as well as a constant thickness denoted by $\delta$. Then the homogenized tensors are given by
	\begin{align*}
	\aahom &=
	\frac{\delta}{12(1-\nu_{12}\nu_{21})}
	\begin{pmatrix}
	E_1 & \nu_{21}E_1 & 0 \\
	\ast & E_2 & 0 \\
	\ast & \ast & 12(1-\nu_{12}\nu_{21}) G
	\end{pmatrix}, \\
	\cchom &=
	\frac{\delta^3}{12(1-\nu_{12}\nu_{21})}
	\begin{pmatrix}
	E_1 & \nu_{21}E_1 & 0 \\
	\ast & E_2 & 0 \\
	\ast & \ast & (1-\nu_{12}\nu_{21}) G
	\end{pmatrix}
	\end{align*}
	and $\bbhom$ vanishes.
\end{lemma}

\begin{remark}
	We remark that due to the orthotropy constraint
	$
	\frac{E_2}{E_1} = \frac{\nu_{21}}{\nu_{12}},
	$
	we can alternatively write $\nu_{12}E_2$ in the second entries in $\aahom,\cchom$ in Lemma \ref{lemma:isotropicPlate}, respectively. 
\end{remark}

The relations in Lemma \ref{lemma:isotropicPlate}, as well as the appearance of the respective entries in the governing macroscopic plate equations, provide us with an intuitive understanding of $\aahom$ and $\cchom$. The entries 
$
\ahom_{1111} \quad\text{and}\quad \ahom_{2222}
$
determine the resistance to applied normal tensional loads, while the ratios
$
\frac{\ahom_{1122}}{\ahom_{1111}} \quad\text{and}\quad \frac{\ahom_{2211}}{\ahom_{2222}}
$
determine the transverse contraction under normal tensional loads.
%
%
\begin{figure}[H]
	\centering
	\includegraphics[trim={6cm 0cm 6cm 0cm}, clip, width=0.7\textwidth]{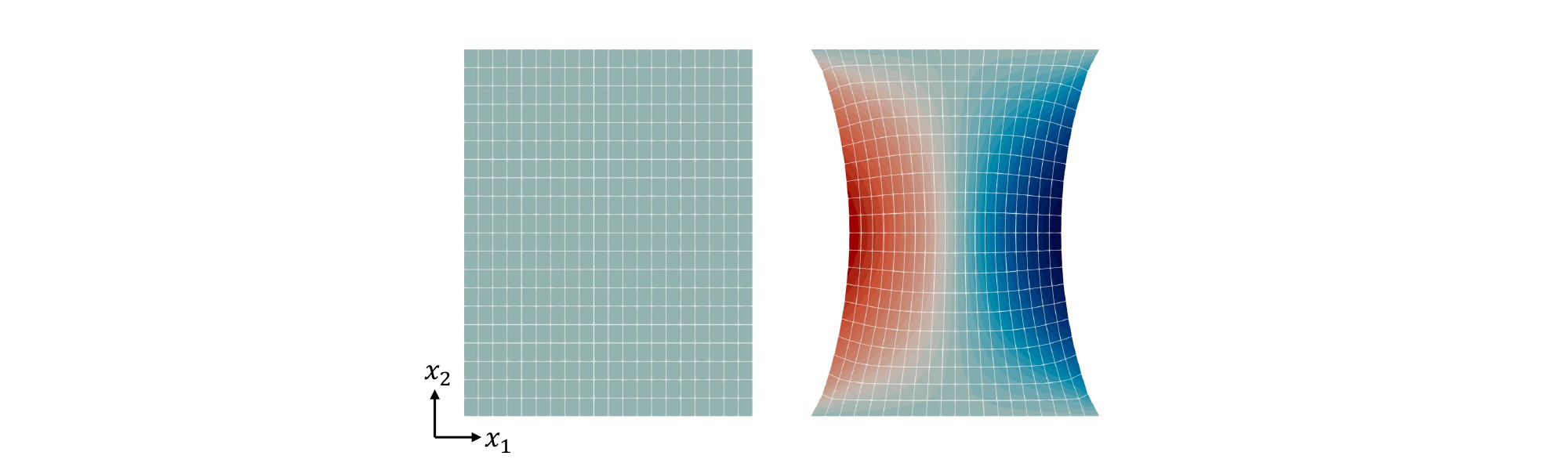}
	\caption[Example for influence of $\aahom$]{Displacement of homogenized textile under applied tension in $x_2$-direction for a zero (left) and a large, non-zero value of $\ahom_{2211}$ (right). Free lateral boundary left and right. Colors indicate displacement in $x_1$-direction.}
	\label{fig:vanishing_aHom}
\end{figure}
 A simulation scenario exemplarily showing the influence of $\ahom_{2211}$ on this Poisson effect is presented in Figure \ref{fig:vanishing_aHom} and a microscopic simulation is shown in Figure \ref{fig:poissonEffectExample}. 
\begin{figure}[H]
	\centering
	\includegraphics[trim={6cm 0cm 6cm 0cm}, clip, width=0.7\textwidth]{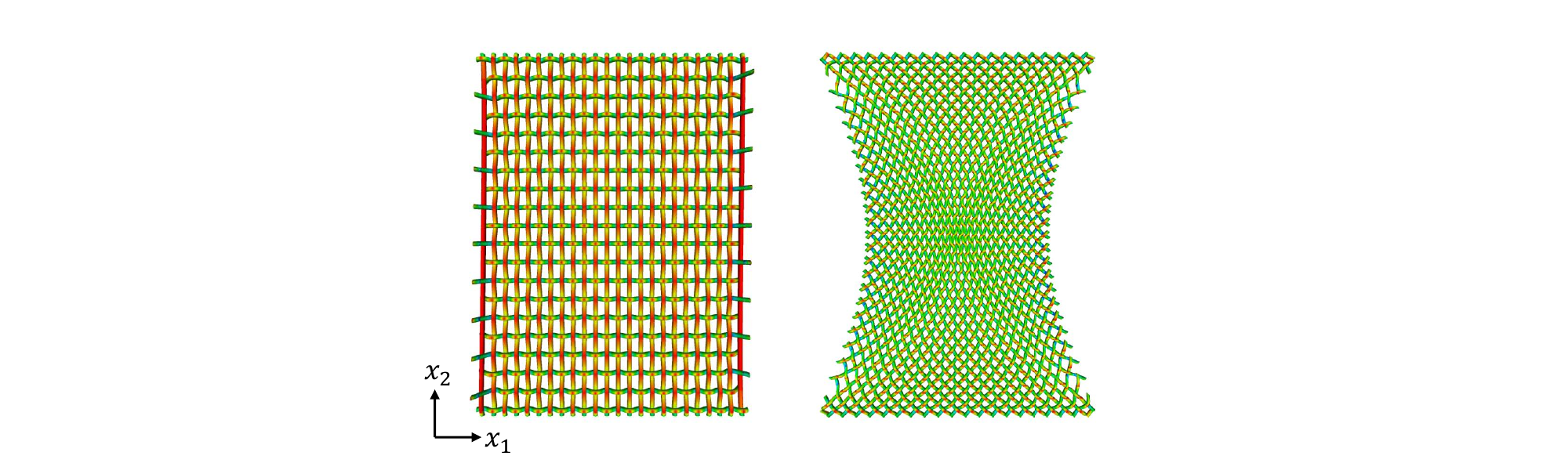}
	\caption[Example for Poisson effect]{Comparison of transverse contraction under applied tension in $x_2$-direction in microscopic simulation for filter sample 1. Yarn orientation along the Cartesian coordinates (left) and rotated by 45° (right). Free lateral boundary left and right with colors indicating local stresses.}
	\label{fig:poissonEffectExample}
\end{figure}
The tensor entry $\ahom_{1212}$ gives a measure of resistance to shearing loads, while the remaining off-diagonal entries $\ahom_{1112},\ahom_{2212}$ introduce a coupling of normal tension and shearing of the textile. An illustrative example is presented in Figure \ref{fig:vanishing_a1222}.
\begin{figure}[H]
	\centering
	\includegraphics[trim={6cm 0cm 6cm 0cm}, clip, width=0.7\textwidth]{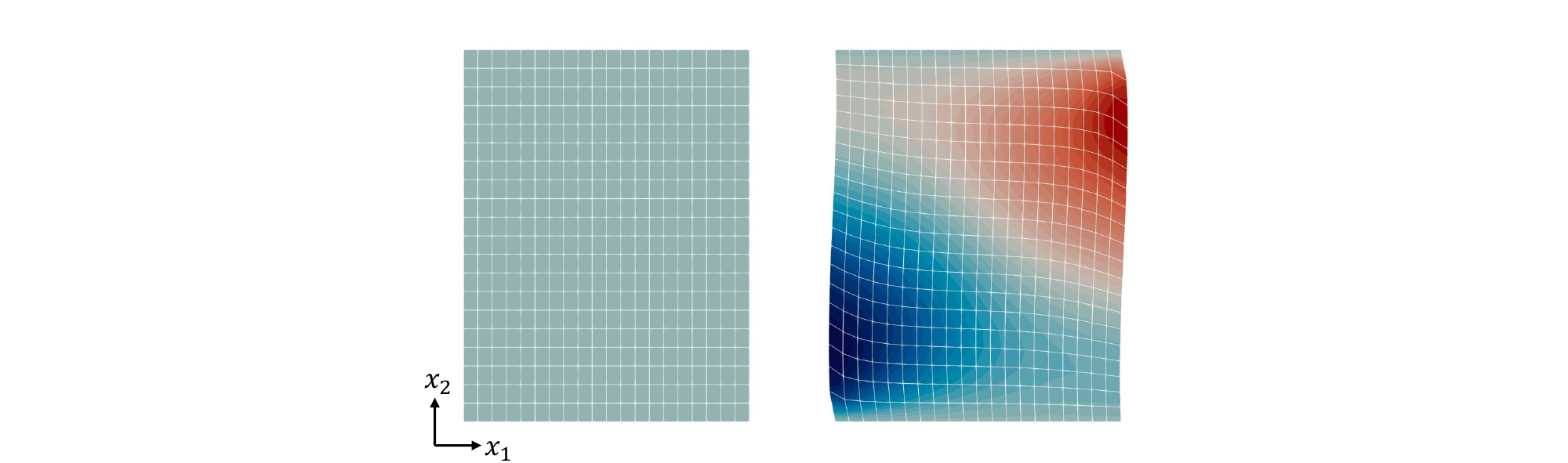}
	\caption[Second example for influence of $\aahom$]{Displacement of homogenized textile under applied tension in $x_2$-direction for a zero (left) and a large, non-zero value of $\ahom_{1222}$ (right). Free lateral boundary left and right. Colors indicate displacement in $x_1$-direction.}
	\label{fig:vanishing_a1222}
\end{figure}

The off-diagonal entries of $\aahom$ may as-well be negative. In case of $\ahom_{1122}$ being negative, one speaks of \textit{auxetic structures}, that expand in transverse direction under normal tensional loads. Changing the sign of $\ahom_{1222}$ in Figure \ref{fig:vanishing_a1222} causes a mirroring of the displacement along the $x_2$-axis.\\
%
%
Similar effective outer-plane bending properties can be formulated for the entries of $\cchom$. Qualitatively, the values 
$
\chom_{1111} \quad\text{and}\quad \chom_{2222}
$
determine the stiffness \wrt normal bending loads, commonly referred to as \textit{flexural rigidity}.\\ The ratios 
$
\frac{\chom_{1122}}{\chom_{1111}}  \quad\text{and}\quad 
\frac{\chom_{2211}}{\chom_{2222}}
$
determine the tendency of transverse bending under normal bending loads, effectively leading to saddle-point formations, \ie hyperbolic paraboloids. 
\begin{figure}[H]
	\centering
	\includegraphics[width=0.9\textwidth]{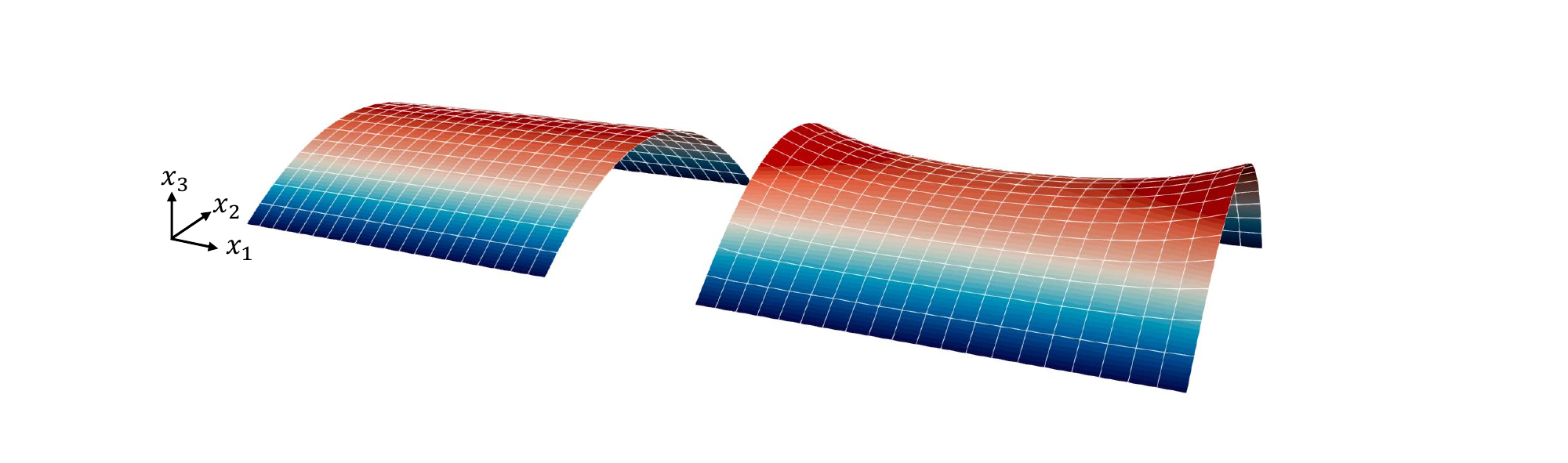}
	\caption[Example for influence of $\cchom$]{Displacement of homogenized textile under applied bending along $x_2$-direction for zero (left) and a large, non-zero value of $\chom_{2211}$ (right). Free lateral boundary left and right. Colors indicate deflection.}
	\label{fig:vanishing_cHom}
\end{figure}
The effect is depicted in Figure \ref{fig:vanishing_cHom} and for a microscopic simulation in Figure \ref{fig:saddlePointExample}. 
\begin{figure}[H]
	\centering
	\includegraphics[trim={0cm 4cm 0cm 4cm}, clip, width=0.7\textwidth]{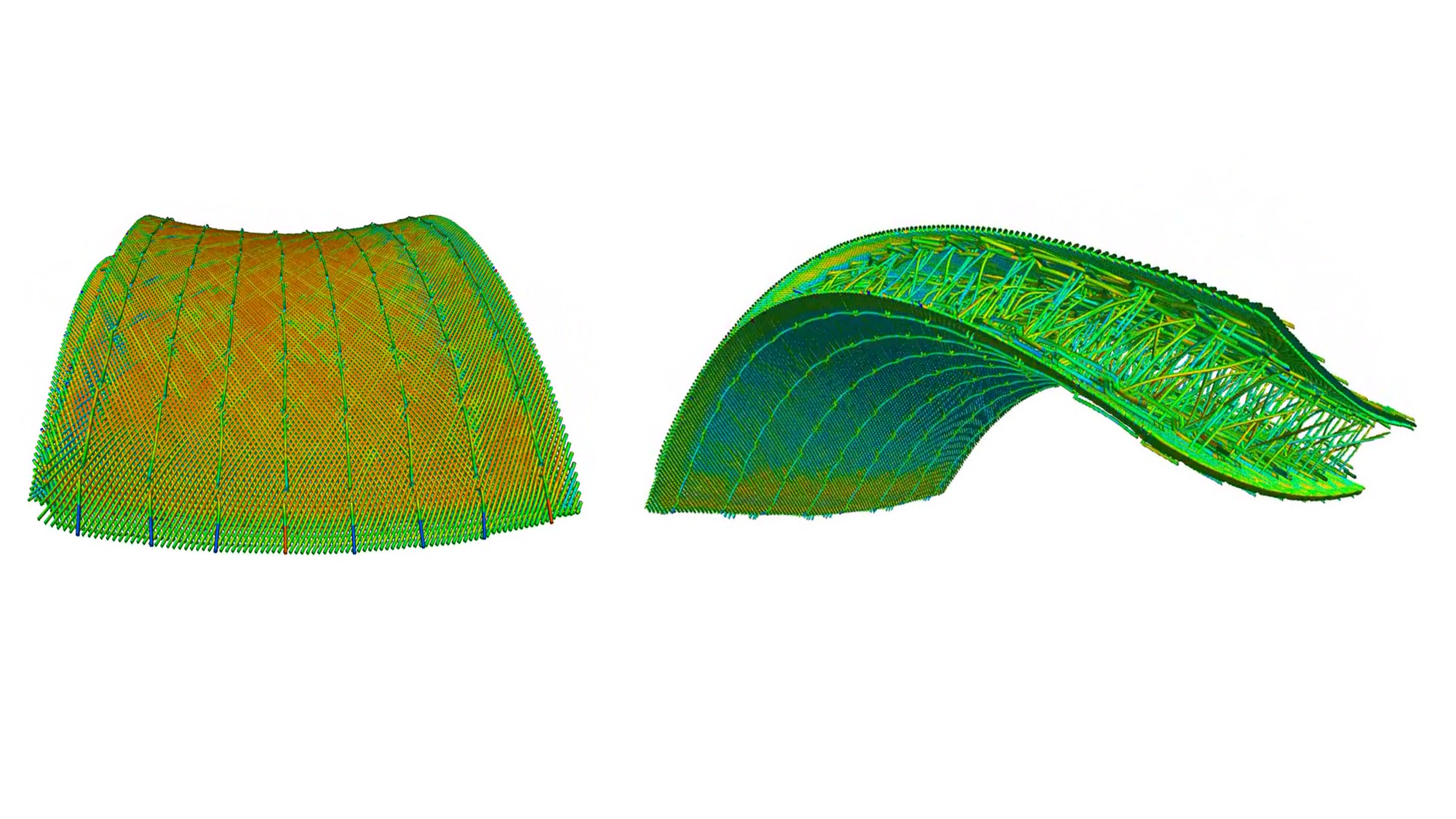}
	\caption[Example for saddle point displacement]{Example for a hyperbolic paraboloid forming under bending of a spacer fabric due to non-zero $\chom_{2212}$. Colors indicate local stresses.}
	\label{fig:saddlePointExample}
\end{figure}
Further, the entry $\chom_{1212}$ is a measure for torsional stiffness, while the remaining off-diagonal entries $\chom_{1112},\chom_{2212}$ introduce an additional coupling between bending and torsion. The coupling effect is demonstrated in Figure \ref{fig:vanishing_c1222}.
\begin{figure}[H]
	\centering
	\includegraphics[width=0.9\textwidth]{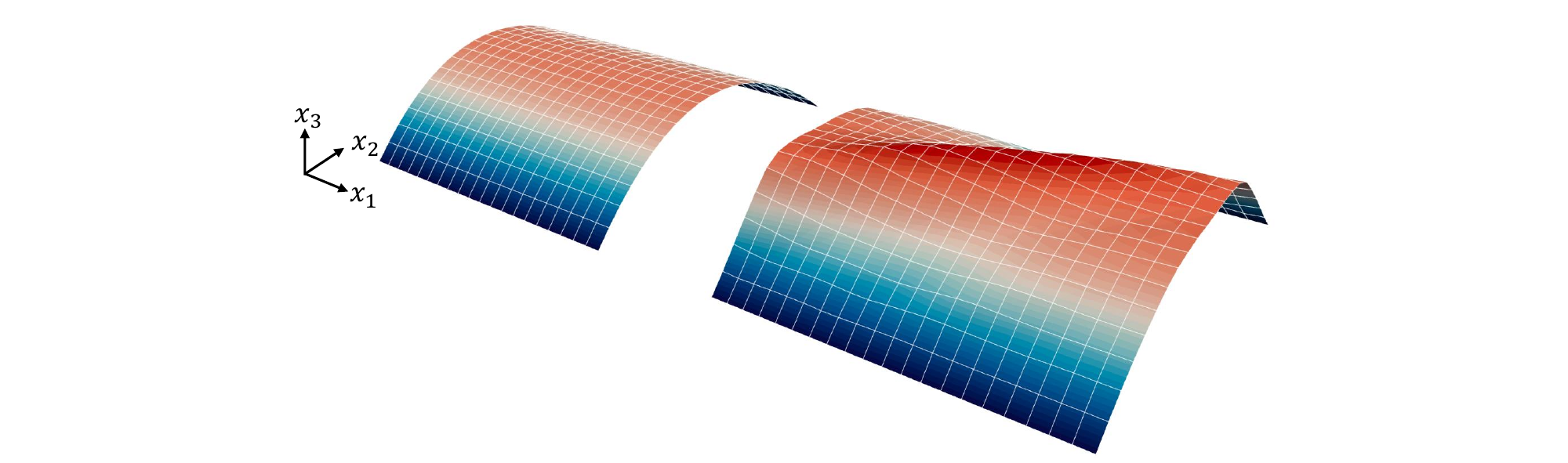}
	\caption[Second example for influence of $\cchom$]{Displacement of homogenized textile under applied bending along $x_2$-direction for zero (left) and a large, non-zero value of $\chom_{2212}$ (right). Free lateral boundary left and right. Colors indicate deflection.}
	\label{fig:vanishing_c1222}
\end{figure}
Similar to $\aahom$, changing the sign of the off-diagonal entries in $\cchom$ results in the inverted transverse bending for $\chom_{2211}$, as well as a mirroring of the displacement in Figure \ref{fig:vanishing_c1222} for $\chom_{2212}$.\\
%
%
From the governing plate equations, we can derive that each entry of $\bbhom$ couples an in-plane strain to a bending moment and vice-versa. 
For illustration, we present the influence of $\bhom_{2211}$ and $\bhom_{2222}$: We consider the case of tension applied in $x_2$-direction which translates to the displacement only in in-plane direction for the case $\bhom_{2211}=\bhom_{2222}=0$. On the other hand, additional coupling with the bending along $x_1$-direction is observed in case of a non-zero value $\bhom_{2211}$, leading to a buckling or wrinkling effect. If $\bhom_{2222}$ is non-zero, the applied strain translates into an additional bending along the $x_2$-direction. All displacements are presented in Figure \ref{fig:vanishing_bHom}. 
\begin{figure}[H]
	\centering
	\includegraphics[width=0.8\textwidth]{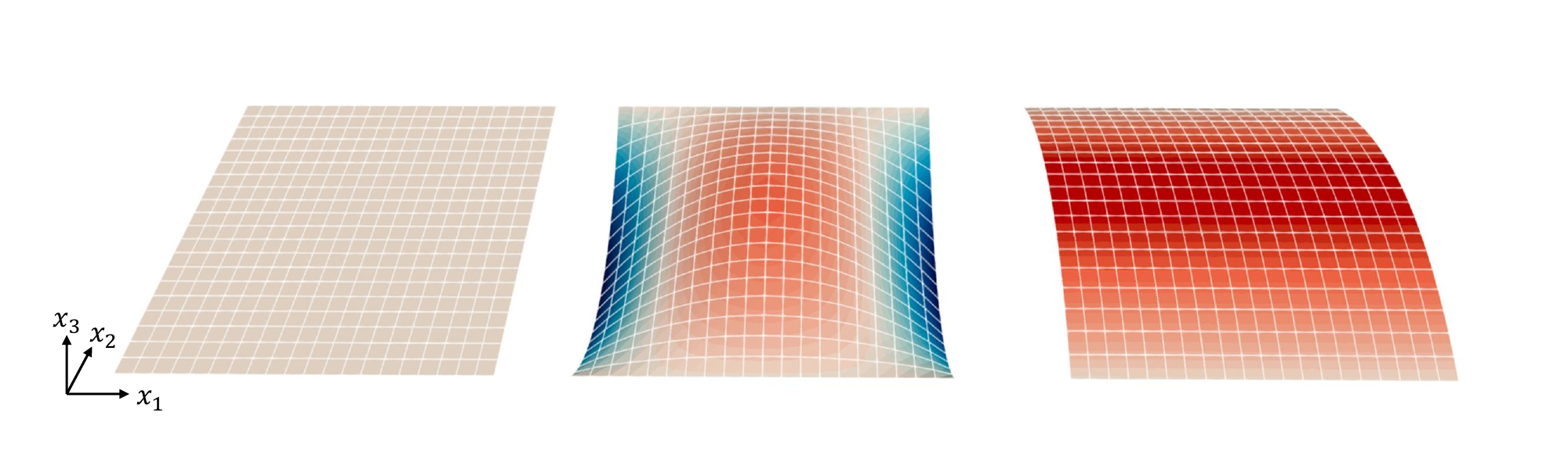}
	\caption[Example for influence of $\bbhom$]{Displacement of homogenized textile under tension in $x_2$-direction for zero (left) and a large, non-zero value of $\bhom_{2211}$ (center), as well as $\bhom_{2222}$ (right). Free lateral boundary left and right. Colors indicate deflection.}
	\label{fig:vanishing_bHom}
\end{figure}
A microscopic simulation with a similar behavior under applied tension in $x_2$-direction is shown in Figure \ref{fig:microSimulationBHom}.
\begin{figure}[H]
	\centering
	\includegraphics[width=0.8\textwidth]{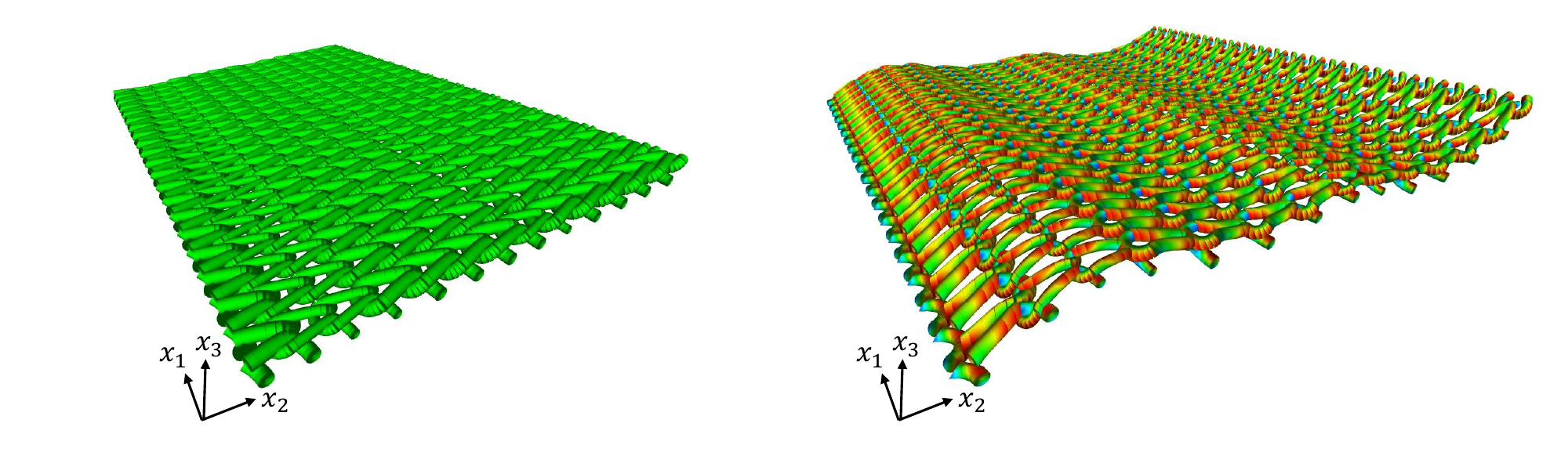}
	\caption[Example for bending under tension]{Bending along $x_2$-direction under applied tension in $x_2$-direction of a weft-knitted textile with alternating yarn material properties in $x_2$-direction. Initial textile (left) and displaced textile (right). Periodic boundary conditions at the lateral boundary. Colors indicate local stresses.}
	\label{fig:microSimulationBHom}
\end{figure}

\subsection{Quantitative description of stiffness tensors}
\label{section:quantitativeDescriptionTensors}

Using the derived qualitative descriptions from the previous section, we discuss the attained simulation results in the following examples.
We emphasize  that if not otherwise stated, all yarns in the examples are oriented along the global coordinate directions, which influences the overall structure of the homogenized stiffness tensors. \\
%
%
\begin{example}\label{example:tensors1}
	For filter sample from Fig. \ref{fig:symmetryExample}, we attain the homogenized tensors
	\begin{align*}
	\aahom &= 
	\begin{pmatrix*}[r]
	\num{4.898e+05}  & \num{2.881e+02}  & \num{-9.978e-01} \\
	\multicolumn{1}{c}{\ast}  & \num{4.898e+05}  & \num{-9.978e-01} \\
	\multicolumn{1}{c}{\ast} & \multicolumn{1}{c}{\ast} & \num{3.258e+04}
	\end{pmatrix*}\si{\newton\per\metre}, \\
	\bbhom &= 
	\begin{pmatrix*}[r]
	\num{1.105e-16} & \num{2.284e-16} & \num{1.613e-15}  \\
	\num{2.224e-17} & \num{-3.845e-16} & \num{-1.631e-15}  \\
	\num{7.391e-17} & \num{-2.481e-17} & \num{4.473e-17}
	\end{pmatrix*}\si{\newton}, \\
	\cchom &= 
	\begin{pmatrix*}[r]
	\num{3.707e-04}  & \num{-1.004e-08} & \num{-8.268e-09} \\
	\multicolumn{1}{c}{\ast} & \num{3.707e-04}  & \num{-8.268e-09} \\
	\multicolumn{1}{c}{\ast} & \multicolumn{1}{c}{\ast} & \num{1.296e-04}
	\end{pmatrix*}\si{\newton \metre}
	\end{align*}
	Qualitatively, these tensor entries are reasonable: We observe the additional symmetry
	\begin{alignat*}{2}
	&\ahom_{1111} = \ahom_{2222}, \quad &&\ahom_{1112} = \ahom_{2212}, \\
	&\chom_{1111} = \chom_{2222}, \quad &&\chom_{1112} = \chom_{2212}
	\end{alignat*}
	stemming from the rotational symmetry of the weave unit itself, as well as vanishing $\bbhom$ entries from Proposition \ref{proposition:vanishingBHom} up to machine precision.\\
%
%
%
%
%
\end{example}

\begin{example}
	With the same parametrization as in Example \ref{example:tensors1} with doubled yarn distance $\Delta_1$, we attain
	\begin{align*}
	\aahom &= 
	\begin{pmatrix*}[r]
	\num{2.448e+05}  & \num{1.407e+02}  & \num[retain-zero-exponent]{-1.161e+00} \\
	\multicolumn{1}{c}{\ast}  & \num{4.900e+05}  & \num{3.586e-01} \\
	\multicolumn{1}{c}{\ast} & \multicolumn{1}{c}{\ast} & \num{1.090e+04}
	\end{pmatrix*}\si{\newton\per\metre}, \\
	\bbhom &= 
	\begin{pmatrix*}[r]
	\num{4.512e-17} & \num{-6.467e-16} & \num{5.309e-16}  \\
	\num{8.081e-18} & \num{2.281e-16} & \num{-2.705e-16}  \\
	\num{3.511e-17} & \num{3.285e-18} & \num{7.555e-18}
	\end{pmatrix*}\si{\newton}, \\
	\cchom &= 
	\begin{pmatrix*}[r]
	\num{1.853e-04}  & \num{-4.767e-09} & \num{-1.575e-09} \\
	\multicolumn{1}{c}{\ast} & \num{3.707e-04}  & \num{-6.919e-09} \\
	\multicolumn{1}{c}{\ast} & \multicolumn{1}{c}{\ast} & \num{9.720e-05}
	\end{pmatrix*}\si{\newton \metre}.
	\end{align*}
	As to be expected, due to the halved yarn density in $x_1$-direction, the overall stiffness for tensional and bending loads in $x_1$-direction is reduced by roughly 50\%, while the stiffness in $x_2$-direction is only slightly affected.
\end{example}

\begin{example}
	Using the same parametrization as for woven filter before, but by rotating the unit cell by 45°, we attain a plain-woven braid as depicted in the right-hand side of Figure \ref{fig:poissonEffectExample}. The extensional stiffness tensor reads
	\begin{align*}
	\aahom &= 
	\begin{pmatrix*}[r]
	\num{3.785e+05}  & \num{2.609e+05}  & \num{-4.130e-11} \\
	\multicolumn{1}{c}{\ast}  & \num{3.785e+05}  & \num{-6.083e-11} \\
	\multicolumn{1}{c}{\ast} & \multicolumn{1}{c}{\ast} & \num{3.874e+05}
	\end{pmatrix*}\si{\newton\per\metre}.
	\end{align*}
	Intuitively, the values $\ahom_{iiii}$ are decreased by the rotation while the ratios $\ahom_{iijj}/\ahom_{iiii}$ became much larger in accordance to the transverse contraction depicted in Figure \ref{fig:poissonEffectExample}. Moreover, the shearing resistance is increased by the new diagonal orientation of yarns.
\end{example}
Intuitively, one expects monotonic dependence of the entries with respect to the structures mass and volume density. Thus, we anticipate a general increasing stiffness for larger yarn diameters and contrarily for smaller yarn distances.\\
%
As it can be seen in the plots of Fig. \ref{fig:dvar}, the dependence on the design is non-linear in general and a quantitative estimate in terms of design parameters proves to be involved even for this relatively simple example. 
\begin{figure}[H]
	\centering
	\includegraphics[scale=0.4]{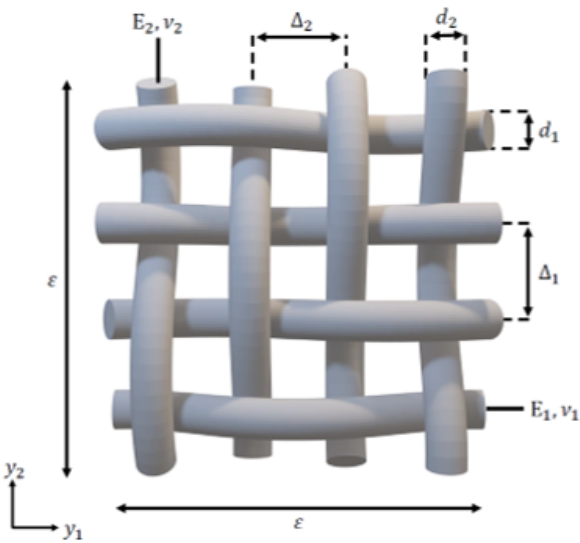}\hspace{0.1cm}
	\includegraphics[scale=0.35]{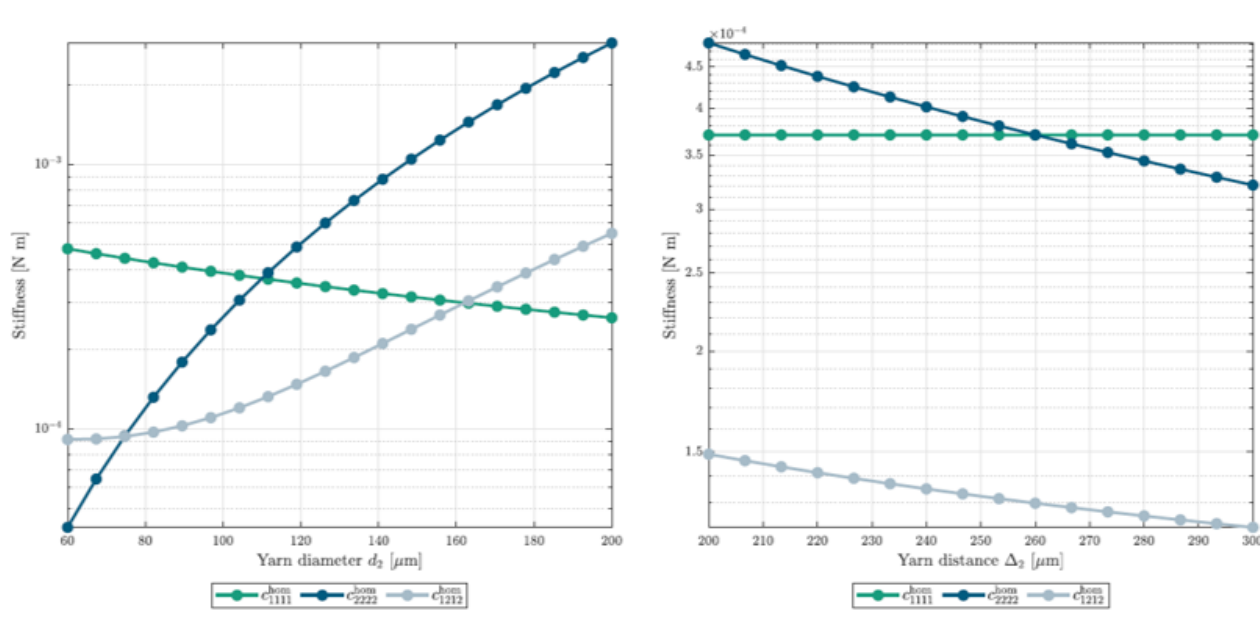}
	\caption{Sensitivity of the bending stiffness $\cchom$ to the yarn thickness and distance between yarns.}
	\label{fig:dvar}
\end{figure}

In the next example,   we consider the variation of the friction coefficient $\gamma_\text{friction}$.
Two extreme cases are displayed in Fig. \ref{fig:fric} left. The first case corresponds to a woven with loose contact, \cite{loose}), while the second case is in the framework of analysis in \cite{GOW}. For a low friction coefficient, we observe rotation at contact points, while for a high friction coefficient, a stiff contact yarns keep their original orientation and just slide in the plane at the contact points slightly. However, in order to see these effects, one should place a changing from the Dirichlet to Neumann boundary conditions on a part of the boundary (they are not such visible in the periodic problems, just in numerical values for the effective coefficients, see Fig. \ref{fig:fric} right). On the left of Fig. \ref{fig:fric}, we fixed the textile plate at a right lower corner and impose symmetry boundary conditions at the left and upper boundaries. 
\begin{figure}[H]
	\centering
	\includegraphics[scale=0.42]{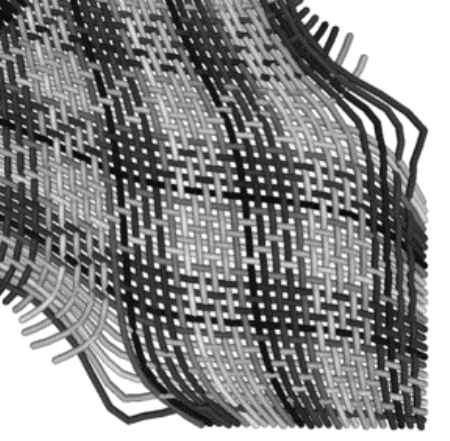}\hspace{0.1cm}
	\includegraphics[scale=0.42]{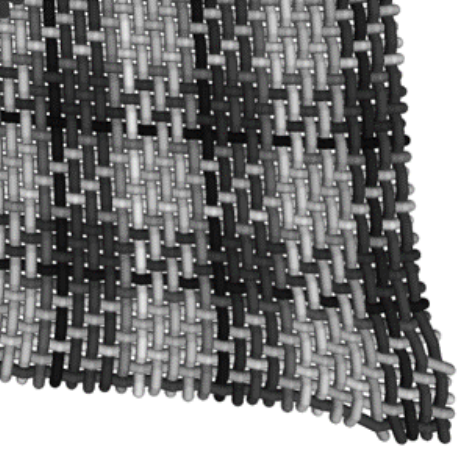}\hspace{0.1cm}
\includegraphics[scale=0.37]{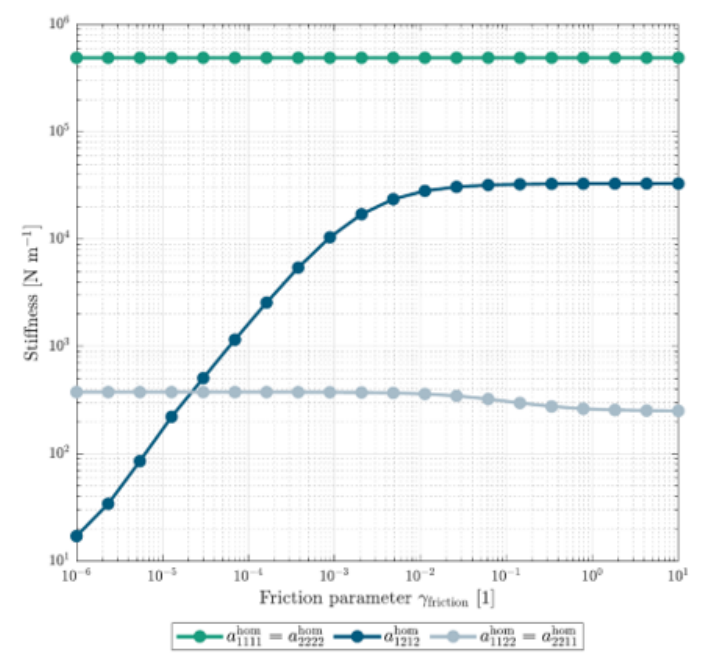}
	\caption[Sensitivity of extensional stiffness to geometric design]{Influence of the friction between yarns. Left: sliding contact $\gamma_{friction}<<1$, vs. stick contact $\gamma_{friction}=1$. The right figure demonstrates changing $\aahom$ for varying friction parameters.}
	\label{fig:fric}
\end{figure}
The effect on the shearing stiffness is visualizedin in the right Fig. \ref{fig:fric}. \\
As a next example, we construct a filter sample, for which the coupling stiffness tensor $\bbhom$ does not vanish. We accomplish this by introducing an asymmetry to a given filter, see Figure \ref{fig:unsym}. \\

For a value of $d_1=\SI{200}{\micro\metre}$, we exemplarily attain
\begin{align*}
\aahom &= 
\begin{pmatrix*}[r]
\num{5.541e+05}  & \num{8.308e+02}  & \num[retain-zero-exponent]{-1.818e+00} \\
\multicolumn{1}{c}{\ast}  & \num{4.737e+05}  & \num{2.038e-01} \\
\multicolumn{1}{c}{\ast} & \multicolumn{1}{c}{\ast} & \num{4.134e+04}
\end{pmatrix*}\si{\newton\per\metre}, \\
\bbhom &= 
\begin{pmatrix*}[r]
\num{5.388e-02} & \num{3.098e-05} & \num{-6.733e-07}  \\
\num{3.006e-05} & \num{5.388e-02} & \num{-2.273e-06}  \\
\num{-3.582e-07} & \num{7.436e-08} & \num{4.226e-03}
\end{pmatrix*}\si{\newton}, \\
\cchom &= 
\begin{pmatrix*}[r]
\num{4.686e-04}  & \num{-3.628e-09} & \num{-4.755e-09} \\
\multicolumn{1}{c}{\ast} & \num{3.878e-04}  & \num{-7.670e-09} \\
\multicolumn{1}{c}{\ast} & \multicolumn{1}{c}{\ast} & \num{1.509e-04}
\end{pmatrix*}\si{\newton \metre}.
\end{align*}
Due to the introduced asymmetry, the off-diagonal entries of $\aahom$ are slightly increased.
We analyze the influence of the choice of $d_1$ on the entries of $\bbhom$ in Figure \ref{fig:bhom_diameter}, restricted to the diagonal entries as well as the entries $\bhom_{1122},\bhom_{2211}$. 
For all entries, we observe a minimal value of \num{0} (up to machine precision) for the fully symmetric case $d_1=\SI{110}{\micro\meter}$ as expected. Even for small deviations around this value, the coupling stiffness tensor attains noticeably larger values.

\begin{figure}[H]
	\centering
	\includegraphics[scale=0.3]{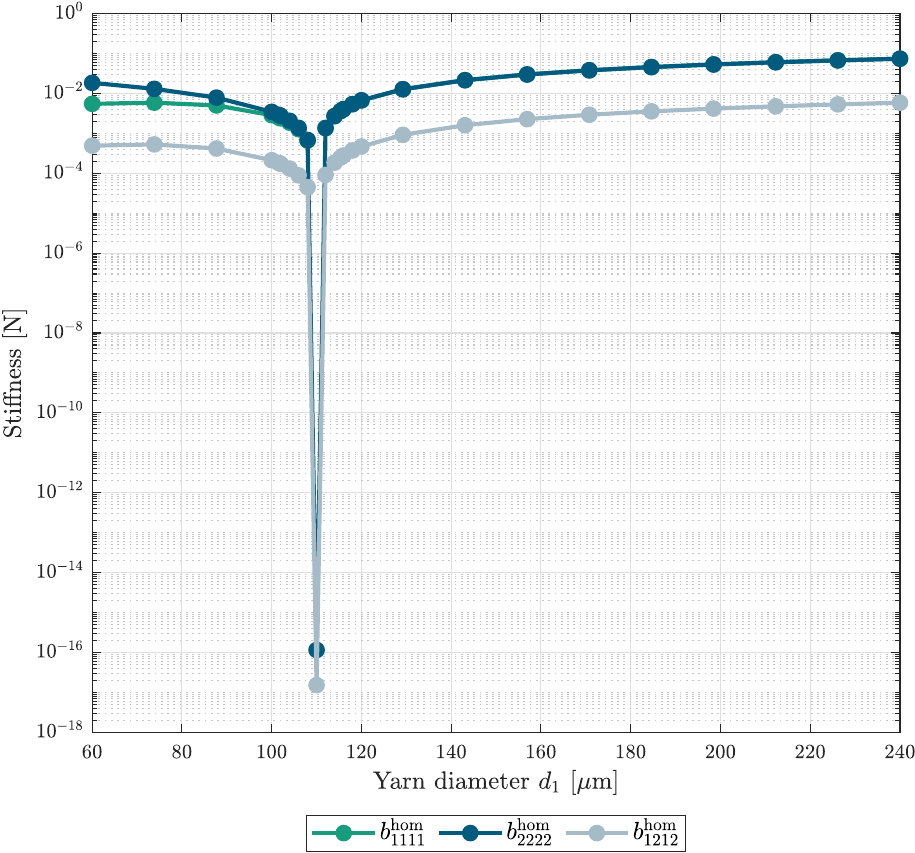}\hspace{0.02cm}
	\includegraphics[scale=0.3]{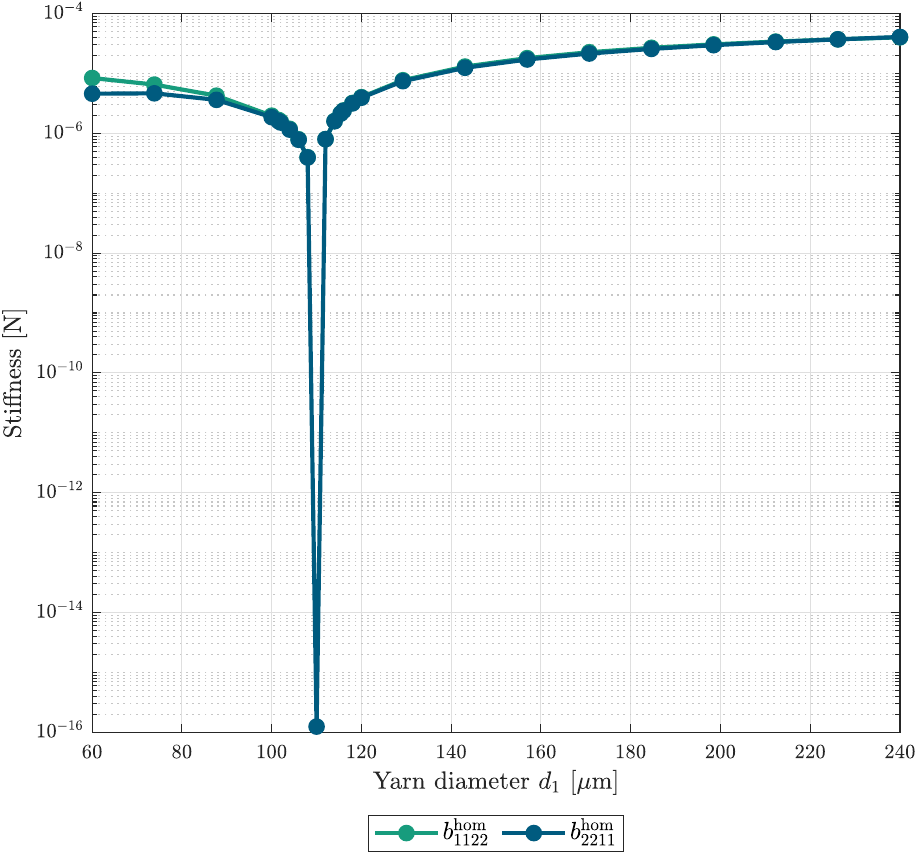}
	\caption[Sensitivity of stiffness for asymmetric filter]{Entries of $\bbhom$ for the asymmetric filter sample as functions of the yarn diameter $d_1$.}
	\label{fig:bhom_diameter}
\end{figure}
Moreover, the attained curves are non-symmetric \wrt the axis $d_1=\SI{110}{\micro\meter}$ with \eg the values $\bhom_{1111},\bhom_{2222}$ being very close to each other for larger $d_1$, while there is a clear difference for smaller $d_1$. Nevertheless, a general monotonic dependence of the entries on the deviation from the fully symmetric case can be observed at least for the investigated parameter range.\\
However, in general, such a symmetry perturbation does not influence the filer symmetry, as it can be seen in Fig. \ref{fig:unsym}.
\begin{figure}[H]
	\centering
	\includegraphics[scale=0.41]{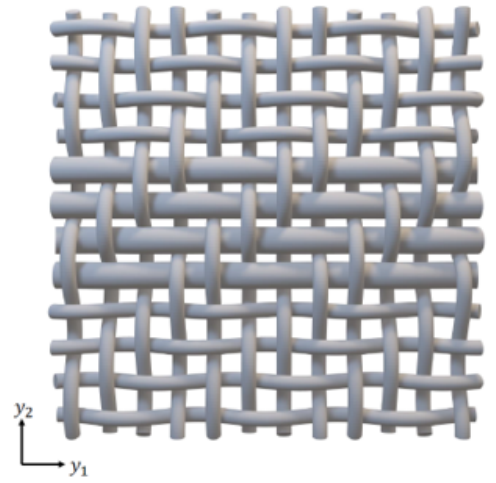}\hspace{0.01cm}
	\includegraphics[scale=0.46]{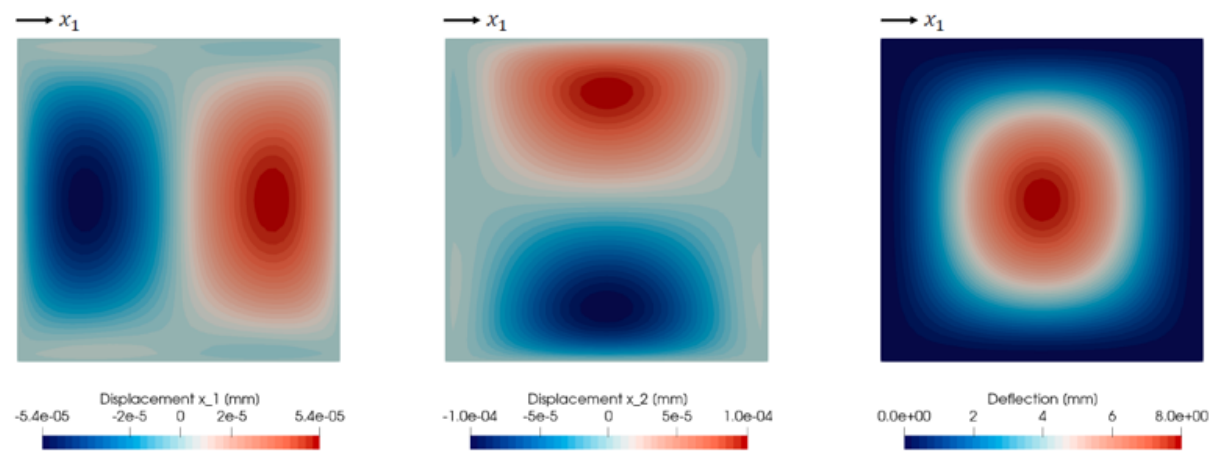}
	\caption[Sensitivity of extensional stiffness to geometric design]{Special, symmetry-perturbing textile design vs. in-plane and normal plate-displacements.}
	\label{fig:unsym}
\end{figure}
Also the fluid velocity profile does not see the local non-symmetry, see Fig. \ref{fig:unsym2}.
\begin{figure}[H]
	\centering
	\includegraphics[scale=0.6]{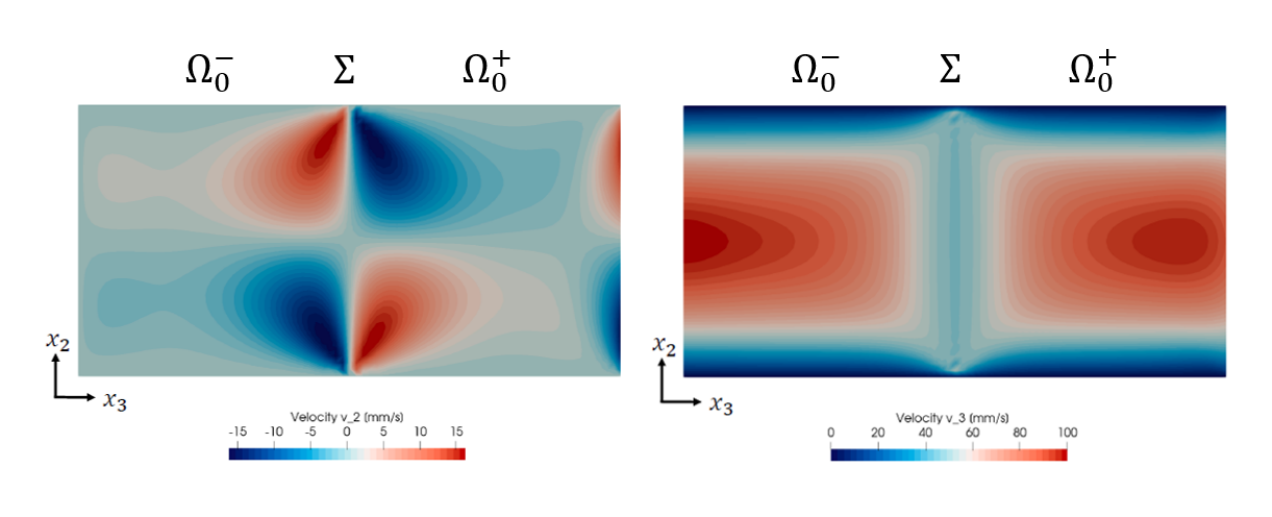}\hspace{0.1cm}
	\caption[Sensitivity of extensional stiffness to geometric design]{Tangential and normal fluid velocity components.}
	\label{fig:unsym2}
	\end{figure}

Finally, we present the filter evolution from the 2D-3D FSI-coupling, where the plate bending is relaxed by the permeability on the first image of Fig. \ref{fig:evol}, and, then on the second image the plate is bent completely. \\
\begin{figure}[H]
	\centering
	\includegraphics[scale=0.45]{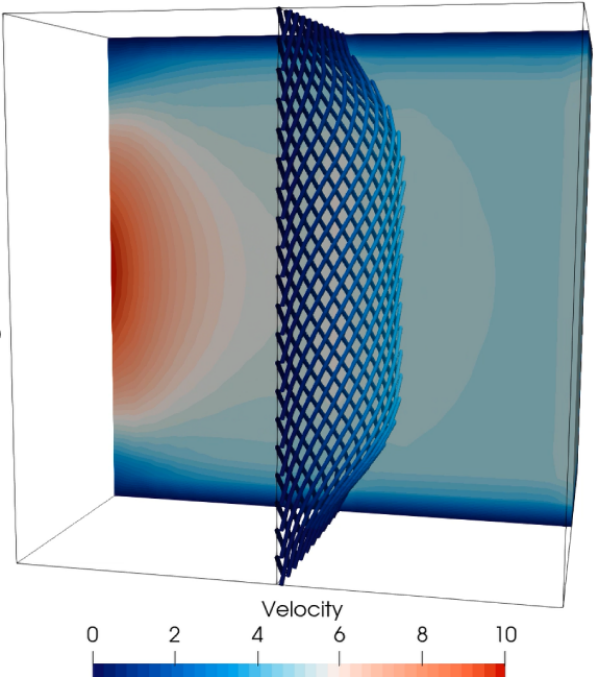}\hspace{0.02cm}
		\includegraphics[scale=0.45]{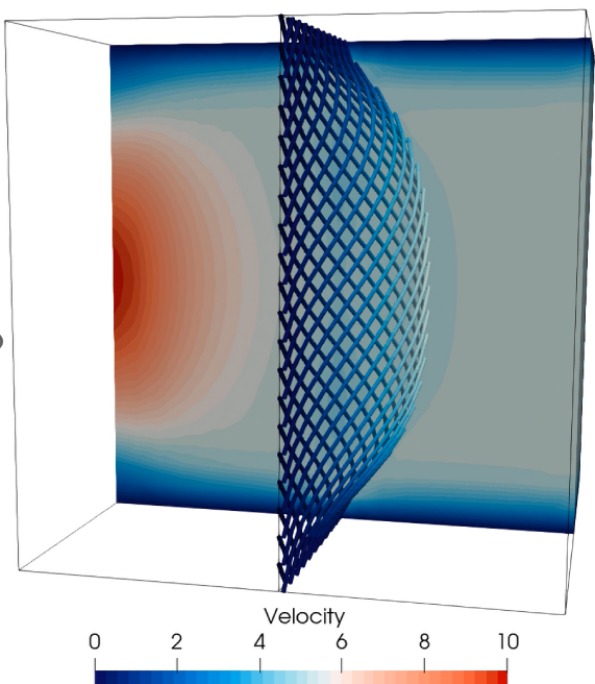}\hspace{0.02cm}
	\caption[Sensitivity of extensional stiffness to geometric design]{Normal fluid velocity evolution, coupled with the permeable flexible 2D-plate.}
	\label{fig:evol}
\end{figure}



\subsection{Extension with anisotropic model parameters}

The macroscopic FSI model remains well-posed if we switch from constant macroscopic model parameters to ones which posses $L^\infty$-regularity on $\Sigma$. 
To ensure existence of solutions, we additionally require coercivity of the permeability tensor almost everywhere on $\Sigma$, as well as coercivity of the bilinear form associated to the homogenized stiffness tensors.

Such formulations are expected to arise \eg if we loosen the periodicity assumption of the microscopic structure to domains with sufficiently \textit{regular} changing structure, see \eg  \cite[Section 5 of Chapter 3]{Panasenko:quasiPeriodic} for so called \textit{quasi-periodic structures} in homogenization of linear elasticity.

For these types of structures, we expect the same form of the cell problems and averaging of the cell solution, however, with an additional dependence on the in-plane variable $\barx$, which results in potentially infinitely many cell problems to solve. For numerical methods, we circumvent this difficulty by considering a spatial discretization of $\Sigma$, solve the cell problems for each grid-point and perform spatial interpolation afterwards. \\

Application examples in mind cover \eg multilayered structures composed of different textiles, as well as patchwork-like fabrics with alternating weaving patterns. 
For illustration, we consider flow through a 2/2 twill woven filter with alternating pattern. One can imagine the structure as a woven filter with equally spaced, parallel, densely woven stripes along the $x_2$-direction, serving as additional support structures. We choose the distances $\Delta_1=\Delta_2=\SI{2.6e2}{\micro \metre}$, as well as the diameter $d_1=\SI{1.6e2}{\micro \metre}$. The diameter $d_2$ is alternating: 24 adjacent yarns have the diameter $\SI{4e1}{\micro\metre}$, followed by 24 yarns with diameter $\SI{1.1e2}{\micro\metre}$ and again 24 yarns with diameter $\SI{4e1}{\micro\metre}$. The resulting periodic unit thereby consists of 18 twill weave units and is illustrated in Figure \ref{fig:anisotropic_structure}. It is periodically repeated 10 times in $x_1$-direction and 180 times in $x_2$-direction, such that the attained filter is quadratic with edge lengths $L_1=L_2=\SI{1.872e2}{\milli\metre}$. We set $L_3=2L_1$.
\begin{figure}[H]
	\centering	
	\includegraphics[trim={0cm 5cm 0cm 5cm}, clip, width=1\textwidth]{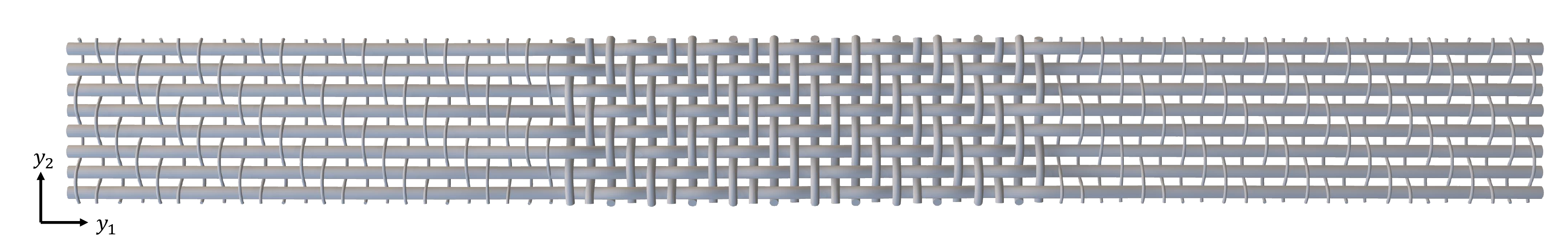}	
	\caption[Filter with alternating yarn diameter]{Periodic unit of considered filter with alternating yarn diameter $d_2$.}
	\label{fig:anisotropic_structure}
\end{figure}

Since the filter is periodic, one may perform the presented homogenization of the entire unit, analogously to the last example in Section \ref{section:quantitativeDescriptionTensors}. In the context of quasi-periodic structures, we homogenize each $4\times 4$ yarn sub-structure to attain piecewise constant homogenized stiffness tensors. The computation of a piecewise constant permeability tensor is performed analogously. \\

For the FSI simulation, we consider the stationary case with Poiseuille profile
\begin{equation*}
\bv^\text{in}(\barx) = v^\text{max}\frac{16}{L_1^2L_2^2} x_1(L_1 - x_1)x_2(L_2 - x_2) \be_3
\end{equation*}
on the inflow boundary, where $v^\text{max}$ is chosen as \SI{5e1}{\milli\metre\per\second}. To resolve the alternating model parameters, we require a significantly finer spatial resolution at the interface compared to the previous examples.

The attained fluid velocity field, as well as the pressure, are presented in Figures \ref{fig:anisotropic_v} and \ref{fig:anisotropic_subdomains}. Due to the smaller permeability in the stripes with larger yarn diameter, the flow mainly passes through the stripes with small yarn diameter. Consequently, on $\Sigma$, both velocity and pressure are oscillating along the $x_1$-direction and remain almost constant along the $x_2$-direction.

\begin{figure}[H]
	\centering
	\includegraphics[width=1\textwidth]{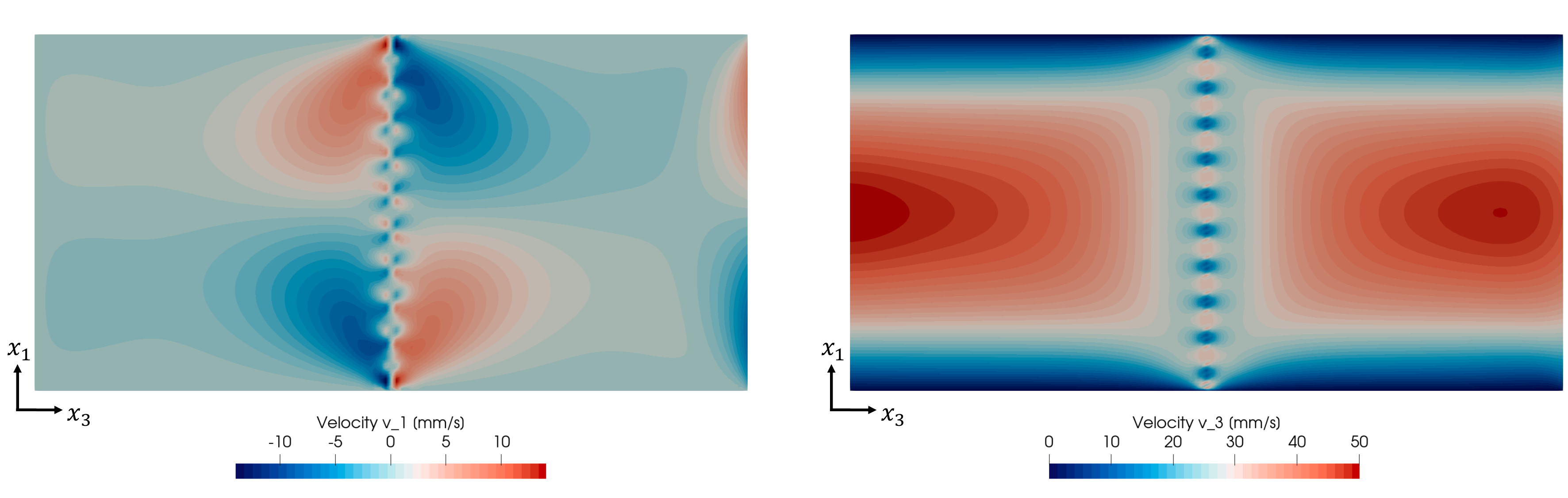}
	\caption[Velocity components for piecewise constant permeability]{Fluid velocity in cross-section of channel for piecewise constant permeability tensor.}
	\label{fig:anisotropic_v}
\end{figure}

\begin{figure}[H]
	\centering
	\includegraphics[width=0.8\textwidth]{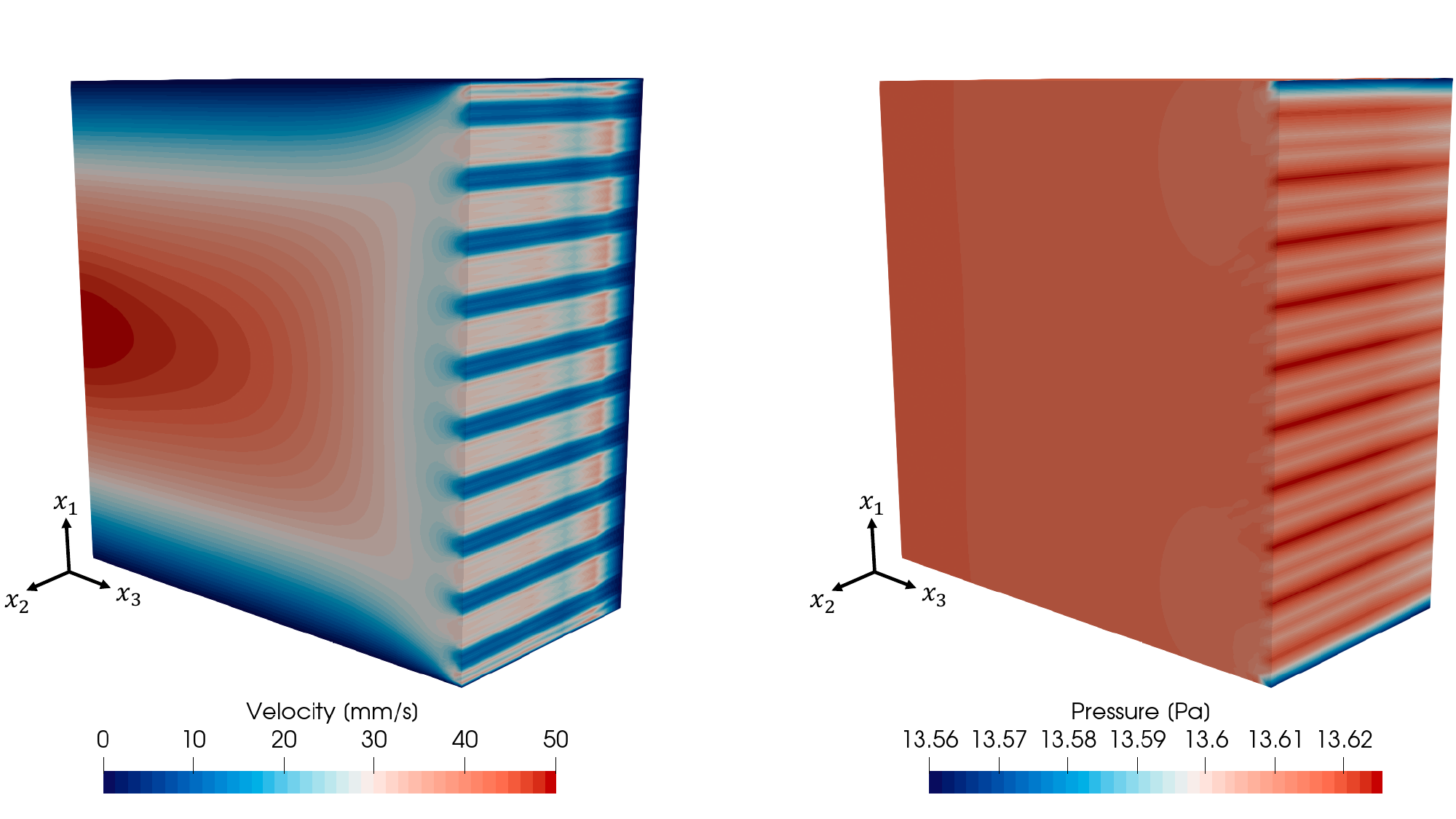}
	\caption[Velocity and pressure for piecewise constant permeability]{Fluid velocity (left) and pressure (right) in cross-section of left subdomain $\O_0^-$ for piecewise constant permeability tensor.}
	\label{fig:anisotropic_subdomains}
\end{figure}

For the structure, the flow-induced displacement profile is shown in Figure \ref{fig:anisotropic_u}. One attains a terraced profile along the $x_1$-direction, while the profile is similar to the previous examples along the $x_2$-direction. Since the jump of fluid stresses is still relatively homogeneous, as can be seen in the small oscillations of the pressure profile in Figure \ref{fig:anisotropic_subdomains}, we can deduce that the terrace effect mainly stems from the alternating bending stiffness. The deduction is confirmed by pure structure simulations with constant right-hand side functions.

\begin{figure}[H]
	\centering
	\includegraphics[width=0.5\textwidth]{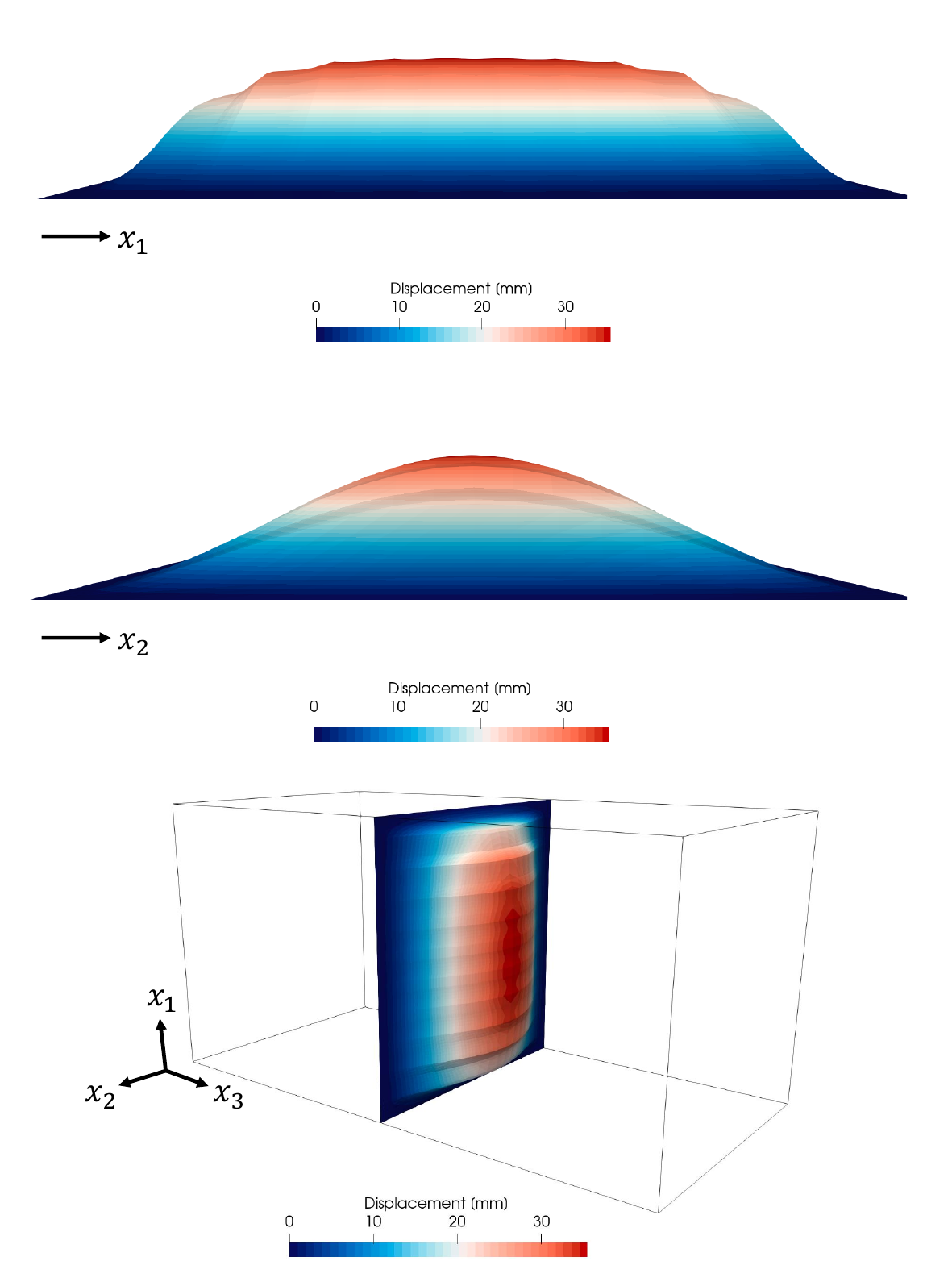}
	\caption[Displacement for piecewise constant stiffness]{Displacement profile for piecewise constant homogenized stiffness tensors.}
	\label{fig:anisotropic_u}
\end{figure}

\section{Conclusions}
\label{sec:conclusions}
This paper offers a very cheep numerical approach, coupling 2D anisotropic porous structural plate with 3D-Stokes flow in a channel by a non-standard interface condition. The interface condition set the pressure jump on the interface proportional to the interface curvature and the interface velocity has a further coupling Darcy-term, mapping the filtration through the Neumann sieve. The problem is solved by $Q_2$ interpolation or bi-spline, or bi-cubic interpolation method, which allows to pre-compute the stiffness matrices in advance and reduce problem to algebraic-differential-system of equations in time. Such a coupling is based on the own asymptotic analysis results, when the plate stiffness is in a contrast with the fluid viscosity.\\
Sec. \ref{section:qualitativeDescriptionTensors} presents the influence of the filter structure on its 2D-plate-coupling with the Stokes 3D-fluid.\\ 
The main result is shown in Fig.\ref{fig:evol}. The time-dependent evolution results into the plate bending under the pressure jump on the interface, then it relaxes by the permeability and then the plate bends completely. We continue the series of the evolution by the back flow reaching the right wall in Fig. \ref{fig:vel} 
\begin{figure}[H]
	\centering
	\includegraphics[scale=0.4]{FSI1.pdf}\hspace{0.02cm}
	\includegraphics[scale=0.4]{FSI2.pdf}
		\includegraphics[scale=0.4]{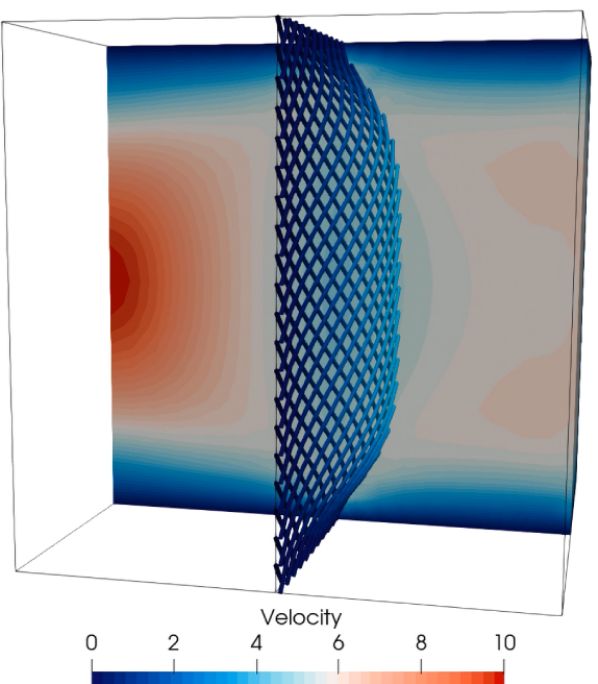}
	\caption[Sensitivity of extensional stiffness to geometric design]{Time-evolution of the normal fluid velocity.}
	\label{fig:vel}
\end{figure}
Fig. \ref{fig:kc} demonstrates two limiting cases, when the plate is impermeable, or rigid, respectively.
\begin{figure}[H]
	\centering
	\includegraphics[scale=0.6]{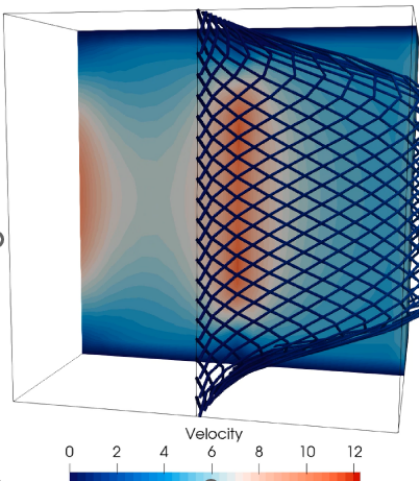}\hspace{0.02cm}
	\includegraphics[scale=0.6]{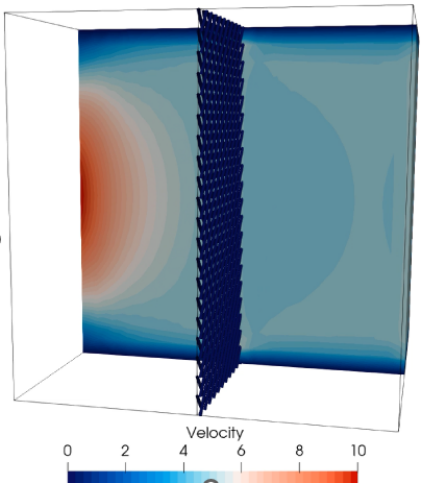}
	\caption[Sensitivity of extensional stiffness to geometric design]{Normal fluid velocity. The first image shows a non-permeable plate with $\hk\to \bm{0}$, while the second one a permeable rigid plate with $\cchom\to \infty$.}
	\label{fig:kc}
\end{figure}

%
%

\section*{Acknowledgments}
The research is funded by the German Research Foundation, DFG-Project OR 190/6-3, on dimension reduction approaches on the yarn level for the FSI with deterministic filter media.

\bibliographystyle{siamplain}
\bibliography{SIAM_Draft_references}

\end{document}